\documentclass[a4paper]{amsart}
\usepackage[all]{xy}
\usepackage[colorlinks=true]{hyperref}
\usepackage{enumerate}
\usepackage{amssymb}
\usepackage{mathtools}
\usepackage{comment}

\usepackage{pstricks-add}
\usepackage{float}
\usepackage{color}
\usepackage{xcolor}
\usepackage{framed}

\usepackage{aurical}
\usepackage{graphicx}
\usepackage{mathrsfs}
\usepackage[cal=boondox]{mathalfa}

\newtheorem{thm}{Theorem}[section]

\newtheorem{prop}[thm]{Proposition}
\newtheorem{lem}[thm]{Lemma}
\newtheorem{cor}[thm]{Corollary}

\theoremstyle{definition}
\newtheorem{defn}[thm]{Definition}
\newtheorem{ex}[thm]{Example}
\theoremstyle{remark}
\newtheorem{rem}[thm]{Remark}

\newcommand{\Rr}{\mathbb R}

\newcommand{\Zz}{\mathbb Z}
\newcommand{\Nn}{\mathbb N}
\newcommand{\Cc}{\mathbb C}

\newcommand{\Kk}{\mathbb K}

\newcommand{\set}[1]{\left\{#1\right\}}
\newcommand{\eval}[1]{\left\langle#1\right\rangle}
\newcommand{\brr}[1]{\left[#1\right]}

\renewcommand{\O}{\ensuremath{\mathcal{O}}}

\newcommand{\ds}{\displaystyle}

\renewcommand{\d}{\mathrm d}                        
\newcommand{\smalcirc}{\mbox{\,\tiny{$\circ $}\,}}  

\newcommand{\al}{\alpha}

\DeclareMathOperator{\ad}{ad}           
\DeclareMathOperator{\End}{End}         
\DeclareMathOperator{\Coder}{Coder}     

\begin{document}

\title{$\O$-operators on  Lie $\infty$-algebras with respect to Lie $\infty$-actions}

\author{Raquel Caseiro}
\address{University of Coimbra\\ CMUC\\ Department of Mathematics\\
Apartado 3008\\
EC Santa Cruz\\
3001-501 Coimbra\\ Portugal
}
\email{raquel@mat.uc.pt}

\author{Joana Nunes da Costa}
\address{University of Coimbra\\ CMUC\\ Department of Mathematics\\
Apartado 3008\\
EC Santa Cruz\\
3001-501 Coimbra\\ Portugal
}
\email{jmcosta@mat.uc.pt}

\thanks{The authors are partially supported by the Centre for Mathematics of the University of Coimbra - UIDB/00324/2020, funded by the Portuguese Government through FCT/MCTES}

\begin{abstract}
 We define $\O$-operators on a Lie $\infty$-algebra $E$ with respect to an action of $E$ on another Lie $\infty$-algebra and we characterize them as Maurer-Cartan elements of a certain Lie $\infty$-algebra obtained by Voronov's higher derived brackets construction. 
 The Lie $\infty$-algebra that controls the deformation of $\O$-operators with respect to a fixed action is determined.
\end{abstract}

\keywords{Lie $\infty$-algebra, $\O$-operator, Maurer Cartan element}

\subjclass{17B10, 17B40, 17B70, 55P43}

\maketitle

\section*{Introduction}             %
\label{sec:introduction}          The first instance of Rota-Baxter operator appeared in the context of associative algebras in 1960, in a paper by Baxter \cite{Baxter}, as a tool to study  fluctuation theory in probability.
Since then, these operators were widely used in many branches of mathematics and mathematical physics.

Almost forty years later, Kupershmidt \cite{K} introduced
$\O$-operators on Lie algebras as a kind of generalization of  classical $r$-matrices, thus opening a broad application of $\O$-operators to integrable systems. Given a Lie algebra $(E,[\cdot,\cdot])$ and a representation $\Phi$ of $E$ on a vector space $V$, an $\O$-operator on $E$ with respect to $\Phi$ is a linear map $T:V \to E$ such that
$[T(x),T(y)]=T\big(\Phi(T(x))(y)- \Phi(T(y))(x)\big)$.
When $\Phi$ is the adjoint representation of $E$, $T$ is a Rota-Baxter operator (of weight zero). $\O$-operators are also called relative Rota-Baxter operators or generalized Rota-Baxter operators.

In recent years Rota-Baxter and $\O$-operators, in different algebraic and geometric settings, have deserved a great interest by  mathematical and physical communities.

 In \cite{TBGS2019}, a homotopy version of $\O$-operators on symmetric graded Lie algebras was introduced. This was the first step towards the definition of an $\O$-operator on a Lie $\infty$-algebra  with respect to a representation on a graded vector space that was given in  \cite{LST2021}.
The current paper also deals with $\O$-operators on Lie $\infty$-algebras, but with a different approach which uses Lie $\infty$-actions instead of representations of Lie $\infty$-algebras.   Our definition is therefore different from  the one given in \cite{LST2021} but there is a relationship between them.

There are two equivalent definitions of Lie $\infty$-algebra structure on a graded vector space $E$, both given by  collections of $n$-ary brackets which are either symmetric or skew-symmetric, depending on the definition we are considering, and must satisfy a kind of generalized Jacobi identities. One goes from one to the other by shifting the degree of $E$ and applying a  \emph{d\'ecalage} isomorphism.  We use the definition in its symmetric version, where the brackets have degree $+1$. Equivalently, this structure can be defined by a degree $+1$ coderivation $M_E$ of $\bar{S}(E)$, the reduced symmetric algebra of $E$,
 such that the commutator $[M_E,M_E]_{c}$ vanishes.

Representations of Lie $\infty$-algebras on graded vector spaces were introduced in \cite{LM}.   In \cite{LST2021}, the authors consider a representation $\Phi$ of a Lie $\infty$-algebra $E$ on a graded vector space $V$ and define an $\O$-operator (homotopy relative Rota-Baxter operator) on $E$ with respect to $\Phi$ as a degree zero element $T$ of $\textrm{Hom}(\bar{S}(V), E)$
satisfying a family of suitable identities.  Inspired by the notion of an action of a Lie $\infty$-algebra on a graded manifold \cite{MZ}, we define an action of a Lie $\infty$-algebra $(E,M_E)$ on a Lie $\infty$-algebra $(V,M_V)$ as a Lie $\infty$-morphism $\Phi$ between $E$ and $\Coder(\bar{S}(V))[1]$, the symmetric DGLA of coderivations of $\bar{S}(V)$. An $\O$-operator on $E$ with respect to the action $\Phi$ is a comorphism between $\bar{S}(V)$ and $\bar{S}(E)$ that intertwines the coderivation $M_E$ and a degree $+ 1$ coderivation of $\bar{S}(V)$ built from $M_V$ and $\Phi$, which turns out to be a Lie $\infty$-algebra structure on $V$ too.

As we said before, the two $\O$-operator definitions, ours and the one in \cite{LST2021}, are different. However, since there is a close connection between Lie $\infty$-actions and representations of Lie $\infty$-algebras, the two definitions can be related.
On the one hand, any representation of $(E,M_E)$ on a complex $(V,\d )$ can be seen as a Lie $\infty$-action of $(E,M_E)$ on $(V,D )$, with $D$ the coderivation given by the differential $\d$, and for this very ``simple" Lie $\infty$-algebra structure on $V$ our $\O$-operator definition  recovers the one given in \cite{LST2021}. On the other hand, any action $\Phi$ of $(E,M_E)$ on $(V,M_V)$ yields a representation $\rho$ on the graded vector space $\bar{S}(V)$ and an $\O$-operator with respect to the action $\Phi$ is not the same as an $\O$-operator with respect to the representation $\rho$. However, there is a way to relate the two concepts.

A well-known  Voronov's construction \cite{V05} defines a Lie $\infty$-algebra structure on an abelian Lie subalgebra $\mathfrak{h}$ of $\Coder(\bar{S}(E\oplus V))$ and  we show that $\O$-operators with respect to  the action $\Phi$ are Maurer-Cartan elements of $\mathfrak{h}$. 

In general, deformations of structures and morphisms are governed by DGLA's or, more generally, by Lie $\infty$-algebras. We do not intend to deeply study   deformations of $\O$-operators on Lie $\infty$-algebras with respect to Lie $\infty$-actions.  Still,  we prove that  deformations of an $\O$-operator are controlled by the twisting of a Lie $\infty$-algebra,  constructed out of a graded Lie subalgebra of $\Coder(\bar{S}(E\oplus V))$.

The paper is organized in four sections. In Section \ref{section1} we collect some basic results on graded vector spaces, graded symmetric algebras and Lie $\infty$-algebras that will be needed along the paper. In Section~\ref{section2}, after recalling the definition of a representation of a Lie $\infty$-algebra on a complex $(V,\d)$  \cite{LM}, we introduce the notion of action of a Lie $\infty$-algebra on another Lie $\infty$-algebra (Lie $\infty$-action) and we prove that a Lie $\infty$-action  of $E$ on $V$ induces a Lie $\infty$-algebra structure on $E\oplus V$. We pay special attention to the adjoint action of a Lie $\infty$-algebra.  In Section \ref{section3} we introduce the {main notion of the paper --}  $\O$-operator on a Lie $\infty$-algebra $E$ with respect to an action of $E$ on another Lie $\infty$-algebra, and we give the explicit relation between these operators and  $\O$-operators on $E$ with respect to a representation on a graded vector space introduced in \cite{LST2021}.
Given an $\O$-operator $T$ on $E$ with respect to a Lie $\infty$-action $\Phi$ on $V$, we show that $V$ inherits a new Lie $\infty$-algebra structure given by a degree $+ 1$ coderivation which is the sum of the initial one on $V$ with a degree $+ 1$ coderivation obtained out of $\Phi$ and $T$. 
We prove that symmetric and invertible comorphisms $T:\bar{S}(E^*) \to S(E)$ are $\O$-operators with respect to the coadjoint action if and only if a certain element of  $\bar{S}(E^*)$, which is defined using the inverse of $T$, is a cocycle for the Lie $\infty$-algebra cohomology of $E$. Section \ref{section3} ends with the characterization of $\O$-operators as Maurer-Cartan elements of a Lie $\infty$-algebra obtained by Voronov's higher derived brackets construction. {The main result in} Section \ref{section4} shows that Maurer-Cartan elements of a {graded Lie subalgebra of}  $\Coder(\bar{S}(E\oplus V))$ encode a  Lie $\infty$-algebra on $E$ and an action of $E$ on $V$. Moreover,  we obtain the Lie $\infty$-algebra that controls the deformation of $\O$-operators with respect to a fixed action.

\section{Lie \texorpdfstring{${{\infty}}$}--algebras} \label{section1}
We begin by reviewing some concepts about graded vector spaces,  graded symmetric algebras and Lie $\infty$-algebras.

\subsection{ Graded vector spaces and graded symmetric algebras}

We will work  with $\Zz$-graded vector spaces with finite dimension over a field $\Kk=\Rr$ or $\Cc$.

Let $E=\oplus_{i\in\Zz}E_i$ be a finite dimensional graded vector space. We call $E_i$ the homogeneous component of $E$ of degree $i$. An  element $x$ of $E_i$ is said to be homogeneous with degree $|x|=i$.
For each $k\in\Zz$, one may shift all the degrees by $k$ and obtain a new grading on $E$. This new graded vector space is denoted by $E[k]$ and is defined by $E[k]_i=E_{i+k}$.

A morphism $\Phi:E\to V$ between two graded vector spaces is a degree preserving  linear map, i.e. a collection of linear maps $\Phi_i:E_i\to V_i$, $i\in\Zz$. We call $\Phi:E\to V$ a  (homogeneous) morphism of degree $k$, for some $k\in\Zz$,   and we write $|\Phi|=k$, if it is a morphism between $E$ and  $V[k]$.
{This way we have a natural grading in the vector space of linear maps between graded vector spaces:
$$
{\mathrm{Hom}}(E,V)=\oplus_{i\in \Zz}{\mathrm{Hom}}_i(E,V).
$$
In particular, ${\mathrm{Hom}}(E,E)=\End(E)=\oplus_{i\in \Zz}\End_i(E) $.}

The dual $E^*$ of $E$ is naturally a graded vector space whose component of degree $i$ is, for all $ i \in {\mathbb Z}$, the dual $(E_{-i})^* $ of $E_{-i}$. In equation: $(E^*)_i = (E_{-i})^*$.

Given two graded vector spaces $E$ and $V$, their direct sum $E\oplus V$ is a vector space with grading
$$
(E\oplus V)_i= E_i\oplus V_i
$$
and their usual tensor product comes equipped with the grading $$(E\otimes V)_i= \oplus_{j+k=i} \, E_j\otimes V_k.$$

We will adopt the Koszul sign convention, for homogeneous {linear maps} $f:E\to V$ and $g:F\to W$ the tensor product $f\otimes g:E\otimes F\to V\otimes W$ is the morphism of degree $|f|+|g|$ given by
\begin{equation*}
(f\otimes g)(x\otimes y)=(-1)^{|x||g|} f(x)\otimes g(y),
\end{equation*}
 for all homogeneous $x\in E$ and $y\in F$.

For each $k\in\Nn_0$, let $T^k(E)=\otimes^k E$,  with $T^0(E)=\Kk$, and let $T(E)=\oplus_{k} T^k(E)$ be  the tensor algebra over $E$. The {\textbf{graded symmetric algebra over $E$}} is the quotient
$$
S(E)=T(E)/\eval{x\otimes y- (-1)^{|x||y|}y\otimes x}.
$$
 The symmetric algebra $S(E)= \oplus_{k\geq0}S^k(E)$ is a graded commutative algebra, whose product we denote by $ \odot $. { For $x=x_1 \odot \ldots \odot x_k \in S^k(E)$, we {set} $|x|= \sum_{i=1}^k |x_i|$.}

For  $n\geq 1$, let $S_n$ be the permutation group of order $n$. For any homogeneous elements $x_1,\ldots,x_n\in E$ and $\sigma \in S_n$, the {Koszul sign} is the element in $\{-1,1\}$ defined by
$$
x_{\sigma(1)} \odot \ldots  \odot x_{\sigma(n)}=\epsilon(\sigma) \, x_1 \odot \ldots \odot x_n.
$$
{As usual, writing $\epsilon(\sigma)$ is an abuse of notation because the Koszul sign  also depends on the $x_i$.}

An element $\sigma$ of $S_{n}$ is called an $(i,n-i)$-unshuffle if  $\sigma(1)<\ldots< \sigma(i)$ and $\sigma(i+1)<\ldots < \sigma(n)$.
The set of $(i,n-i)$-unshuffles is denoted by $Sh(i,n-i)$. { Similarly, $Sh(k_1,\ldots, k_j)$ is the set of $(k_1,\ldots, k_j)$-unshuffles, i.e., elements of $S_n$ with $k_1+\ldots+ k_j=n$ such that the order is preserved within each block of length $k_i$, $1\leq i\leq j$.}

The reduced
 symmetric algebra  $\bar S(E)=\oplus_{k\geq 1}S^k(E)$ has a natural { coassociative and } cocommutative coalgebra structure given by the coproduct $\Delta:\bar S(E)\to \bar S(E)\otimes \bar S(E)$,
 \begin{equation*}
 \Delta(x)=0, \,\, x \in E;
 \end{equation*}
 \begin{equation*}\label{def:coproduct:S(E)}
 \Delta(x_1\odot\ldots \odot x_n)=\sum_{i=1}^{n-1}\sum_{\sigma\in Sh(i,n-i)}\epsilon(\sigma) \, \left(x_{\sigma(1)}\odot\ldots \odot x_{\sigma(i)}\right)\otimes \left(x_{\sigma(i+1)}\odot\ldots\odot x_{\sigma(n)}\right),
 \end{equation*}
for $x_1,\dots,x_n\in E$.

We will  mainly use Sweedler notation: given $x \in \bar S(E)$,  $$\Delta^{(1)}(x)=\Delta(x)=x_{(1)}\otimes x_{(2)},$$
 and {the coassociativity yields} $$\Delta^{(n)}(x)=(\mathrm{id}\otimes\Delta^{(n-1)})\Delta(x)=x_{(1)}\otimes \ldots\otimes x_{(n+1)},\quad n\geq 2.$$ 
Notice that 
 $$\Delta^{(n)}(x)=0, \,\,\,\,  x \in S^{\leq n}(E).$$
 {The cocommutativity of the coproduct is expressed, for homogeneous elements of $\bar S(E)$, as
 $$x_{(1)}\otimes x_{(2)}=(-1)^{|x_{(1)}||x_{(2)}|}x_{(2)}\otimes x_{(1)}.$$
 }

Let  $V$ be another graded vector space. A
 linear map $f:\bar S(E)\to V$ is  given by a collection of maps $f_k:S^k(E)\to V$, $k\geq 1$, and is usually denoted by $f=\sum_kf_k$.

\begin{rem}
Every linear map ${f}:S^k(E)\to V$ corresponds to a graded symmetric $k$-linear map $f \in {\mathrm{Hom}}(\otimes^k E,V)$ through the quotient map $p_k:\otimes^k E \to S^k(E)$ i.e., $f\equiv{f} \smalcirc p_k$. In the sequel, we  shall often write $${f}(x_1 \odot \ldots \odot x_k)=f(x_1, \ldots, x_k), \quad x_i \in E.$$
\end{rem}
A {\bf coalgebra morphism}  {(or {\bf comorphism}) between the coalgebras $(\bar S(E), \Delta_E)$ and $(\bar S(V), \Delta_V)$ is a morphism $F:\bar S(E) \to \bar S(V)$ of graded vector spaces} such that
$$
(F\otimes F)\smalcirc \Delta_E=\Delta_V\smalcirc F.
$$
There is a one-to-one correspondence between coalgebra morphisms $F:\bar S(E)\to \bar S(V)$ and {degree preserving} linear maps $f:\bar S(E)\to V$. {Each $f$ determines $F$ by
\begin{equation*}
F(x)=\sum_{k\geq1}\frac{1}{k!}f(x_{(1)})\odot\ldots\odot f(x_{(k)}), \, x \in \bar S(E),
\end{equation*}
and $f=p_V \smalcirc F$, with $p_V:\bar S(V)\to V$ the projection map.}

A {degree $k$} {\bf coderivation} of $\bar S(E)$,  for some $k \in \mathbb Z$, is a linear map $Q:\bar S(E)\to \bar S(E)$ of degree $k$ such that
$$
\Delta\smalcirc Q=(Q\otimes \mathrm{id} + \mathrm{id}\otimes Q)\smalcirc\Delta.
$$
We also have a one to one correspondence between coderivations of $\bar S(E)$ and linear maps $q=\sum_{i}q_i:\bar S(E)\to E$:

\begin{prop} \label{prop:isomorphism:families:coderivations}
Let $E$ be a graded vector space and  $p_E:\bar S(E)\to E$ the projection map.
For every linear map $q=\sum_iq_i: \bar S(E)\to E$,  the linear map
$\ds Q: \bar S(E)\to  \bar S(E )$
given by
\begin{equation} \label{eq:coderivation}
Q(x_1\odot \ldots \odot x_n)=\sum_{i=1}^n \, \sum_{\sigma \in Sh(i,n-i)}\epsilon(\sigma)q_i\left(x_{\sigma(1)}, \ldots, x_{\sigma(i)}\right)\odot x_{\sigma(i+1)}\odot \ldots \odot x_{\sigma(n)},
\end{equation}
is the unique coderivation of $\bar S(E )$ such that $p_E\smalcirc Q =q$.
\end{prop}

{In Sweedler notation, Equation \eqref{eq:coderivation} is written as:
$$Q(x)= q(x_{(1)}) \odot x_{(2)} + q(x), \,\,\, x \in \bar S(E ).$$
}

When  $E$ is a  finite dimensional graded vector space, we may identify $S(E^*)$ with $(SE)^*$. Koszul sign conventions yield, for each homogeneous elements  $f,g\in E^*$,
\begin{equation*}
(f\odot g)(x\odot y)=(-1)^{|x||g|}f(x)\, g(y) + f(y)\, g(x), \quad x,y\in E.
\end{equation*}

\subsection{Lie \texorpdfstring{$\infty$}--algebras}

 We briefly recall the definition of Lie $\infty$-algebra \cite{LS},  some basic examples and related concepts. 

We will consider the symmetric approach to Lie $\infty$-algebras.

\begin{defn}
A \textbf{symmetric Lie $\infty$-algebra} (or a \textbf{Lie$[1]$ $\infty$-algebra}) is a graded vector space $E=\oplus_{i\in\Zz} E_i$ together with a family of degree $+1$ linear maps $l_k: S^k(E)\to E$, $k\geq 1$, satisfying
\begin{equation}\label{eq:def:symm:L:infty:algebra}
\sum_{i+j=n+1}\sum_{\sigma\in Sh(i,j-1)}\epsilon(\sigma) \, l_j\left(l_i\left(x_{\sigma(1)},\ldots, x_{\sigma(i)}\right),x_{\sigma(i+1)}\ldots, x_{\sigma(n)}\right)=0,
\end{equation}
for all $n\in\Nn$ and all homogeneous elements $x_1,\ldots,x_n\in E$.
\end{defn}

The \emph{d\'ecalage} isomorphism   \cite{V05}
 establishes a one to one correspondence between  skew-symmetric Lie $\infty$-algebra structures $\set{l'_k}_{k \in \Nn}$ on $E$ and symmetric Lie $\infty$-algebra structures $\set{l_k}_{k \in \Nn}$ on $E[1]$:
 {$$l_k(x_1, \ldots, x_k)= (-1)^{(k-1)|x_1|+(k-2)|x_2|+ \ldots + |x_{k-1}|}l'_k(x_1, \ldots, x_k).$$}
 {In the sequel, we frequently write Lie $\infty$-algebra, omitting the term symmetric.}

\begin{ex}[Symmetric graded Lie algebra]
A symmetric graded Lie algebra is a symmetric Lie $\infty$-algebra $E=\oplus_{i\in\Zz}E_i$ such that $l_n = 0$ for $n \neq 2$. Then the degree $0$ bilinear map on $E[-1]$ defined by
 \begin{equation}
 \label{eq:decalage}
  [\![x,y]\!] := (-1)^{i} l_2(x,y), \hbox{ for all $x \in E_i,y\in E_j $,}
  \end{equation}
is a graded Lie bracket.
In particular, if $E=E_{-1}$ is concentrated  on degree $-1$, we get a Lie algebra structure.
\end{ex}

\begin{ex}[Symmetric DGLA algebra]
A symmetric differential graded Lie algebra (DGLA) is a symmetric  Lie $\infty$-algebra $E=\oplus_{i\in\Zz}E_i$  such that $l_n=0$ for $n \neq 1$  and $n \neq 2$.

Then, from \eqref{eq:def:symm:L:infty:algebra}, we have that $d:=l_ 1$ is a degree $+1$ linear map $d:E\to E$  squaring zero and satisfies the following compatibility condition with the bracket $\brr{\cdot , \cdot }:=l_2(\cdot , \cdot)$ :
\begin{equation*}
\label{eq:gradedLie}
\left\{\begin{array}{l}
d\brr{x,y} + \brr{d(x),y} + (-1)^{|x|}\brr{x, d(y)}=0,\\
\brr{\brr{x,y},z} + (-1)^{|y||z|}\brr{\brr{x,z},y} + (-1)^{|x|}\brr{x,\brr{y,z}}=0,
\end{array}\right.
\end{equation*}
{Applying the \emph{d\'ecalage} isomorphism, $(E[-1], d, [\![\cdot, \cdot]\!])$ is a (skew-symmetric) DGLA, with $[\![\cdot, \cdot]\!]$ given by \eqref{eq:decalage}.}
\end{ex}

\begin{ex} \label{example:DGLA_Coder}
\label{ex:DGLA:End:E}
Let $(E=\oplus_{i\in\Zz}E_i, \d)$ be a cochain complex. 
Then ${\End (E)[1]}$ 
 has a natural symmetric DGLA  structure with $l_1=\partial, \;\;l_2=\brr{\cdot , \cdot}$ given by:
\begin{equation*}
\left\{\begin{array}{l}
 \partial\phi
=-\d \smalcirc \phi + (-1)^{|\phi|+1} \phi\smalcirc \d,\\
 \brr{\phi,\psi}
=(-1)^{|\phi|+1}\left(\phi\smalcirc \psi - (-1)^{(|\phi|+1)(|\psi|+1)} \psi\smalcirc\phi \right),
\end{array}\right.
\end{equation*}
for $\phi, \psi$ homogeneous elements of $\End(E)[1]$. {In other words, $ \partial\phi=- [ \d, \phi]_{c}$ and $\brr{\phi,\psi}=(-1)^{\deg(\phi)}[ \phi, \psi]_{c}$, with $[\cdot, \cdot]_{c}$ the graded commutator on $\End(E)$ and $\deg(\phi)$ the degree of $\phi$ in $\End(E)$.}
\end{ex}

The symmetric Lie bracket $\brr{\cdot ,\cdot}$ on $\End(\bar S(E))[1]$  preserves  $\Coder(\bar S(E))[1]$, the space of coderivations of $\bar S(E)$, {so that $(\Coder(\bar S(E))[1], \partial, \brr{\cdot ,\cdot})$ is a symmetric DGLA.}

The isomorphism between ${\mathrm{Hom}}(\bar S(E),E)$ and $\Coder(\bar S(E))$ given by Proposition \ref{prop:isomorphism:families:coderivations},
induces a Lie bracket on  ${\mathrm{Hom}}(\bar S(E),E)$ known as the Richardson-Nijenhuis bracket:
$$
\brr{f,g}_{_{RN}}(x)=f(G(x))-(-1)^{|f||g|}g(F(x)),\quad x\in \bar S(E),
$$
for each $f,g\in  {\mathrm{Hom}}(\bar S(E),E)$, where $F$ and $G$ denote the coderivations defined by $f$ and $g$, respectively. {In other words, $\brr{F,G}_{c}$ is the (unique) coderivation of $\bar S(E)$ determined by $\brr{f,g}_{_{RN}} \in {\mathrm{Hom}}(\bar S(E),E)$.

Degree $+1$ elements $l:=\sum_k l_k$ of ${\mathrm{Hom}}(\bar S(E),E)$ satisfying $[l,l]_{_{RN}}=0$ define a Lie $\infty$-algebra structure on $E$.} This way we have an alternative definition of Lie $\infty$-algebra \cite{LS}:
\begin{prop}
A Lie $\infty$-algebra is a graded vector space $E$ equipped with a degree $+1$ coderivation $M_E$ of $\bar S(E)$ such that \begin{equation*}\brr{M_E,M_E}_{c}=2M_E^2=0.\end{equation*}
\end{prop}

The dual of the coderivation $M_E$ yields a differential $\d_*$ on $\bar{S}(E^*)$. 
The \textbf{cohomology of the Lie $\infty$-algebra} $\left( E, M_E \equiv \set{l_k}_{k\in\Nn}
\right)$ is the cohomology defined by the differential $\d_*$.

A {\bf Maurer-Cartan element} of a Lie $\infty$-algebra $(E,\set{l_k}_{k\in\Nn})$ is a degree zero element $z$ of $E$ such that
\begin{equation} \label{def:MC:element}
\sum_{k \geq 1}\frac{1}{k!}\, l_{k}(z, \ldots,z) =0.
\end{equation}
The set of Maurer-Cartan elements of $E$ is denoted by $\textrm{MC}(E)$.
Let $z$ be a Maurer-Cartan element of $(E,\set{l_k}_{k\in\Nn})$ and set, for $k\geq 1$,
\begin{equation} \label{def:twisting:MC}
l_k^z(x_1, \ldots, x_k):= \sum_{i\geq 0}\frac1{i!}\, l_{k+i}(z, \ldots, z, x_1, \ldots, x_k).
\end{equation}
Then, $(E,\set{l_k^z}_{k\in\Nn})$ is a Lie $\infty$-algebra, called the \emph{twisting of} $E$ by $z$ \cite{G}.
 For filtered, or even weakly filtered Lie $\infty$-algebras, the convergence of the infinite sums defining Maurer-Cartan elements and twisted Lie $\infty$-algebras (Equations (\ref{def:MC:element}) and (\ref{def:twisting:MC})) is guaranteed (see \cite{G,FZ,LST2021}).

For a symmetric graded Lie algebra $(E,l_2)$, the twisting by $z \in\textrm{MC}(E)$ is the symmetric DGLA $(E, l_1^z=l_2(z, \cdot),l_2^z=l_2)$.


\subsection{Lie \texorpdfstring{$\infty$}--morphisms}

A morphism of Lie $\infty$-algebras is a  morphism between symmetric coalgebras that is compatible with the Lie $\infty$-structures.

\begin{defn}\label{defn:Lie:infty:morphism}
Let $(E,\set{l_k}_{k\in\Nn})$ and $(V,\set{m_k}_{k\in\Nn})$ be Lie $\infty$-algebras. A \textbf{Lie ${\infty}$-morphism } $\ds \Phi:E \rightarrow V$
is given by a collection of degree zero linear maps:
$$
\Phi_k:S^k(E)\to V,\quad k\geq 1,
$$
such that, for each $n\geq 1$,
\begin{align}\label{eq:def:Lie:infty:morphism}
&\sum_{\begin{array}{c} \scriptstyle{k+l=n}\\ \scriptstyle{\sigma\in Sh(k,l)}\\ \scriptstyle{l\geq 0, \, k\geq 1}\end{array}}\!\!\!\!\!\!\! \varepsilon(\sigma) \Phi_{1+l}\Big(l_k(x_{\sigma(1)},\ldots, x_{\sigma(k)}), x_{\sigma(k+1)}, \ldots, x_{\sigma(n)}\Big)  \\ &=\!\!\!\!\!\!\!\!\!\!\! \sum_{\begin{array}{c}\scriptstyle{k_1+\ldots+ k_j=n} \\ \scriptstyle{\sigma\in Sh(k_1,\ldots, k_j)}\end{array}} \!\!\!\!\!\!\!  \frac{\varepsilon(\sigma)}{j!}\,m_j\Big(\Phi_{k_1}(x_{\sigma(1)}, \ldots x_{\sigma(k_1)}),
\Phi_{k_2}(x_{\sigma(k_1+1)}, \ldots \nonumber x_{\sigma(k_1+k_2)}),\ldots,\\
&\hspace{6cm}  \Phi_{k_j}(x_{\sigma(k_1+\ldots+k_{j-1}+1)}, \ldots, x_{\sigma(n)})\Big),\nonumber
\end{align}
\end{defn}

\noindent{If $\Phi_k=0$ for $k\neq 1$, then $\Phi$ is called a \textbf{{strict} Lie ${\infty}$-morphism.}}

A \textbf{curved Lie $\infty$-morphism} $E\to V$, with $V$ a weakly filtered Lie $\infty$-algebra, is a degree zero linear map $\Phi:S(E)\to V$ satisfying, for $n\geq 0,$ an adapted version  of \eqref{eq:def:Lie:infty:morphism} where the indexes $k_1,\ldots, k_j$ on the right hand side of the equation run from $0$ to $n$.  The zero component $\Phi_0:\Rr\to V_0$ gives rise to an element $\Phi_0(1)\in V_0$, which by abuse of notation we denote by $\Phi_0$. The curved adaptation  of \eqref{eq:def:Lie:infty:morphism}, for $n=0$, then reads $0=\sum_{j\geq 1}\frac{1}{j!}\,m_j(\Phi_0,\ldots,\Phi_0)$. In other words,  $\Phi_0$ is a Maurer Cartan element of $V$ \cite{MZ}.

\

Considering the coalgebra morphism $\Phi: \bar S(E)\to \bar S(V)$ defined by
the collection of degree zero linear maps $$
\Phi_k:S^k(E)\to V,\quad k\geq 1,
$$
  we see that Equation \eqref{eq:def:Lie:infty:morphism} is equivalent to
$\Phi$ preserving the Lie $\infty$-algebra structures:
$$\Phi \smalcirc M_E=M_V\smalcirc \Phi.$$

\section{Representations of Lie \texorpdfstring{$\infty$}--algebras} \label{section2}

 A complex $(V,\d)$ induces a  natural symmetric DGLA structure in $\End(V)[1]$, see Example \ref{ex:DGLA:End:E}.

\begin{defn}
 A \textbf{representation} of a Lie $\infty$-algebra $(E,\set{l_k}_{k\in\Nn})$ on a complex $(V,\d)$ is a Lie $\infty$-morphism $$\Phi:(E,\set{ l_k}_{k\in\Nn})\rightarrow (\End(V)[1],\partial, \brr{\cdot , \cdot}),$$
 {i.e., $\Phi \smalcirc M_E= M_{\End(V)[1]}\smalcirc \Phi$, where $M_E$ is the coderivation determined by $\sum_k l_k$ and $ M_{\End(V)[1]}$ is the coderivation determined by $\partial+ \brr{\cdot , \cdot}$.}
\end{defn}

 Equivalently, a representation of $E$ is defined by a collection of degree $+1$ maps
$$\Phi_k:S^k(E)\to \End(V),\quad k\geq 1,$$
such that, for each $n\geq 1$,
\begin{align} \label{eq:def:representation}
\lefteqn{\sum_{\begin{array}{c} \scriptstyle{i=1}\\ \scriptstyle{\sigma\in Sh(i,n-i)}\end{array}}^{\scriptstyle{n}}\!\!\!\!\!\!\!\!\!\! \varepsilon(\sigma)\Phi_{n-i+1}\left(l_i\left(x_{\sigma(1)}, \ldots, {x_{\sigma(i)}}\right), {x_{\sigma(i+1)}}, \ldots, {x_{\sigma(n)}}\right)=} \\
&= &  \partial  \Phi_n(x_1,\ldots, x_n)+ \frac{1}{2} \!\!\!\!\!\!\!\!\!\!\sum_{\begin{array}{c} \scriptstyle{j=1}\\ \scriptstyle{\sigma\in Sh(j,n-j)}\end{array}}^{\scriptstyle{n-1}} \!\!\!\!\!\!\!\!\!\!\varepsilon(\sigma)\brr{\Phi_j({x_{\sigma(1)}}, \ldots, {x_{\sigma(j)}}) , \Phi_{n-j}({x_{\sigma(j+1)}}, \ldots, {x_{\sigma(n)}}) }. \nonumber
\end{align}

\begin{rem}
A representation on a complex $(V,\d)$ can be seen as a curved Lie $\infty$-morphism $\Phi:E \to \End (V)[1]$,  with $\Phi=\sum_{k\geq 0}\Phi_k$ and  $\Phi_0=\d$.  In fact, the first term on the right hand-side of Equations \eqref{eq:def:representation} is given by
$$\partial  \Phi_n(x_1,\ldots, x_n)=\brr{\Phi_0, \Phi_n(x_1,\ldots, x_n)},$$
and we have a curved Lie $\infty$-morphism $$\Phi:(E,\set{ l_k}_{k\in\Nn})\rightarrow (\End(V)[1],\brr{\cdot, \cdot})$$
between the Lie $\infty$-algebra $E$ and the symmetric graded Lie algebra $(\End(V)[1],\brr{\cdot, \cdot})$ (see \cite{MZ}, Lemma 2.5).
{This is why  sometimes a representation of a Lie $\infty$-algebra $E$ on a complex $(V,\d)$  is called a representation on the graded vector space $V$ (compatible with the differential $\d$ of $V$). }
\end{rem}

Any representation {$\Phi:E\to \End(V)[1]$} of a Lie $\infty$-algebra $E$ on a complex $(V,\d)$ has a dual one. 
Let $$
{}^*:\End(V)\to \End(V^*)
$$
be the Lie $\infty$-morphism given by
\begin{equation}\label{eq:dual:map}
\eval{f^*(\al),v}=-(-1)^{|\al||f|}\eval{\al, f(v)},\quad f\in\End(V), \al\in V^*, v\in V.   \end{equation}
The \textbf{dual representation} ${} ^*\Phi:E\to \End(V^*)[1]$ is obtained by composition of   $\Phi$ with this Lie $\infty$-morphism. It is a representation on the complex $(V^*,\d^*)$
 and is given by
\begin{equation}\label{eq:dual:representation}
\eval{{}^*\Phi(e)(\al), v}=-(-1)^{(|e|+1)|\al|}\eval{\al,\Phi(e)(v)}, \quad e\in \bar S(E),\, \al\in V^*, \, v\in V.  \end{equation}

\begin{rem} \label{rem:lie:infty:algebra:E+V}
Given a representation $\Phi:E\rightarrow \End(V)[1]$ on a complex $(V,\d)$, defined by the collection of degree $+1$ linear maps $\Phi_k:S^k(E)\to \End(V)$, $k\geq 1$, one may consider the collection of degree $+1$ maps
$\phi_k:S^{k}(E)\otimes V\to V$, $k\geq 0$, where  $\phi_0=\d:V\to V$ {and $\phi_k(x,v)=(\Phi_k(x))(v), \, k\geq1$.}

{The embedding $\bar S(E)\oplus \big(S(E)\otimes V\big) \hookrightarrow \bar S(E\oplus V)$,  provides a collection of maps}  
\begin{equation*}
 \tilde\Phi_k: S^{k}(E\oplus V)\to E\oplus V, \,\,\, k\geq 1,
\end{equation*}
given by
\begin{align*}
&\tilde\Phi_k\left((x_1, v_1), \ldots, (x_{k}, v_{k}) \right)  \\& =\bigg(l_{k}(x_1, \ldots, x_{k}), \sum_{i=1}^{k}(-1)^{|x_i|(|x_{i+1}|+\ldots+|x_{k}|)}\phi_{k-1}(x_1,\ldots, \widehat{x_i},\ldots, x_{k},v_i)\bigg),
\end{align*}
and  we may
express Equations \eqref{eq:def:representation} as
\begin{equation}\label{eq:def:representation:LModule}
\tilde\Phi_\bullet\left(\tilde\Phi_\bullet(x_{(1)})\odot x_{(2)} \right) + \tilde\Phi_1\tilde\Phi_\bullet(x)=0, \quad x\in \bar S(E\oplus V).
\end{equation}

Equation \eqref{eq:def:representation:LModule} means that $\tilde{\Phi} $  equips  $E \oplus V$ with a Lie $\infty$-algebra structure.

\end{rem}

Now suppose the graded vector space $V$ has  a Lie $\infty$-algebra structure $\set{ m_k}_{k\in\Nn}$ given by a coderivation $M_V$ of $\bar S(V)$. By the construction in Example~\ref{example:DGLA_Coder}, the
coderivation  $M_V$ of $\bar S(V)$ defines
a symmetric DGLA structure in $\Coder(\bar S(V))[1]$:
$$
\partial_{M_V}Q=- M_V\smalcirc Q + (-1)^{\deg(Q)} Q\smalcirc M_V,
$$
$$
\brr{Q,P}=(-1)^{\deg(Q)}\bigg(Q\smalcirc P-(-1)^{\deg(Q)\deg(P)}P\smalcirc Q\bigg),
$$
where $\deg(Q)$ and $\deg(P)$ are the degrees of $Q$ and $P$ in $\Coder (\bar S(V))$.

{Generalizing the notion of an action of a graded Lie algebra on another graded Lie algebra,} we have the following definition of an action of a Lie $\infty$-algebra  on another Lie $\infty$-algebra:

\begin{defn}
An \textbf{action of the Lie $\infty$-algebra}  $(E,M_E\equiv\set{ l_k}_{k\in\Nn})$ on the Lie $\infty$-algebra $(V,M_V\equiv\set{ m_k}_{k\in\Nn})$, or a \textbf{Lie $\infty$-action} of $E$ on $V$, is a  Lie $\infty$-morphism
$$
\Phi: (E,\set{ l_k}_{k\in\Nn}) \to (\Coder (\bar S(V))[1], \partial_{M_V}, \brr{\cdot,\cdot}).
$$
\end{defn}

\begin{rem}  \label{rem:maps:definition:action}
{Being a Lie $\infty$-morphism,} an action $$
\Phi: E \to \Coder (\bar S(V))[1]
$$ is univocally defined by a collection of {degree $+1$} linear maps
$$
\Phi_k: S^k(E)\to \Coder {(\bar S(V)), \quad k \geq 1.}
 $$
 By the isomorphism provided in Proposition \ref{prop:isomorphism:families:coderivations}, {and since each $\Phi_k(x), \,\, x \in S^k(E)$, is a coderivation of $\bar S(V)$,} we see that an action is
 completely defined by a collection of  linear maps
\begin{equation}\label{eq:linear:maps:representation}
   \Phi_{k,i}:  S^k(E)\otimes  S^i(V)\to V, \quad i,k\geq 1. 
\end{equation}
{We will denote the coderivation $\Phi_k(x)$ simply by $\Phi_x$.}
\end{rem}

\begin{rem}
If we define $\Phi_{0}:=M_{V}$, then
an action is equivalent to  a curved Lie $\infty$-morphism  between $E$ and the graded Lie algebra $\Coder (\bar S(V))$ (compatible with the Lie $\infty$-structure in $V$) \cite{MZ}. {In this case, $\Phi=\sum_{k\geq 0}\Phi_k$ is called a {\bf{curved Lie $\infty$-action}}}.
\end{rem}

{There is a close relationship between representations and actions on Lie $\infty$-algebras.}

 First notice that each linear map $\mathcal{l}:V\to V$ induces a (co)derivation of $\bar S(V)$. Hence we may see $\End (V)[1]$ as a Lie $\infty$-subalgebra of  $\Coder(\bar S(V))[1]$. Therefore,
given a representation $\Phi:E\to \End(V)[1]$ of the Lie $\infty$-algebra $E$ on the complex $(V,\d)$, we have a natural action of $E$ on the Lie $\infty$-algebra $(V,M_{V})$, where $M_{V}$ is the coderivation defined by the map $\d:V\to V$.
In this case, we say {\bf{the action is induced by a representation}.}

Moreover, for each action $\Phi:E\to \Coder{(\bar S(V))}[1]$ of $E$ on the  Lie $\infty$-algebra $(V, M_{V}\equiv\set{ m_k}_{k\in\Nn})$, we have a  representation of $E$ on $V$ given by
 the collection of maps $\Phi_{k,1}:S^k(E)\otimes V \to V$, $k\geq 1$, or equivalently, $\Phi_{k,1}\equiv \rho_k:S^k(E)\to \End(V)$, $k\geq 1$. The morphism $\rho=\sum_k \rho_k$ is a representation of the Lie $\infty$-algebra $E$ on the complex $(V, \d=m_1)$,  called the \textbf{linear representation defined by $\Phi$}.

Finally one should notice that, given a Lie $\infty$-algebra $(V,M_V)$, the graded vector space $\Coder (\bar S(V))[1]$ is a Lie $\infty$-subalgebra of $\End(\bar S(V))[1]$. Therefore, any action  $\Phi:E\to \Coder(\bar S(V))[1]$ of the Lie $\infty$-algebra $E$ on $(V,M_V)$ yields a representation of $E$ on the graded vector space $\bar S(V)$. We call it the \textbf{representation induced by the action  $\Phi$}. The coderivation $M_V$ defines a (co)derivation of $\bar S(\bar S(V))$ and the representation is  compatible with this (co)derivation.

\begin{rem} In \cite{MZ}, the authors define an action of a finite dimensional Lie $\infty$-algebra $E$ on a graded manifold $\mathcal{M}$ as a Lie $\infty$-morphism $\Phi:E\to \mathfrak{X}({\mathcal{M}})[1]$.  As the authors point out, when $\mathcal{M}$ is the graded manifold defined by a  finite dimensional Lie $\infty$-algebra, we have an action of a Lie $\infty$-algebra on another Lie $\infty$-algebra.  The definition presented here is a particular case of theirs because we are only considering coderivations of $\bar S(V)$, i.e. coderivations of $S(V)$  vanishing on the field $S^{0}(E)$.  This restrictive case reduces to the usually Lie algebra action on another Lie algebra (and its semi-direct product) while the definition given in \cite{MZ}, gives rise to general Lie algebra extensions. For our purpose, this definition is more adequate. \end{rem}

Next, with the identification $S^n(E\oplus V)\simeq \oplus_{k=0}^n S^{n-k}(E)\otimes S^{k}(V)$, we see that the action $\Phi$ determines a coderivation of $\bar S(E\oplus V)$. Together with $M_E$ and $M_V$ we have a Lie $\infty$-algebra structure on $E\oplus V$.
Next proposition can be deduced from \cite{MZ}.
\begin{prop}\label{prop:lie:infty:algebra:E+V} Let $(E, M_E\equiv\set{ l_k}_{k\in\Nn})$ and $(V,M_V\equiv\set{ m_k}_{k\in\Nn})$ be Lie $\infty$-algebras.
An action
$$\Phi: E \to \Coder (\bar S(V))[1]$$
defines a Lie $\infty$-algebra structure on $E\oplus V$.
\end{prop}

\begin{proof}
We consider the brackets $\{\mathfrak{l}_n\}_{n \in \Nn}$ on $E\oplus V$ given by:
\begin{align*}
& \mathfrak{l}_n(x_1, \ldots, x_n)=l_n(x_1, \ldots, x_n), \quad x_i \in E\\
& \mathfrak{l}_n(v_1, \ldots, v_n)=m_n(v_1, \ldots, v_n), \quad v_i \in V \\
&  \mathfrak{l}_{k+n}(x_1, \ldots, x_k,v_1, \ldots, v_n )= \Phi_{k,n}(x_1, \ldots, x_k,v_1, \ldots, v_n ),
\end{align*}
with  $\Phi_{k,n}:S^k(E)\otimes S^n(V) \to V$ the collection of linear maps defining $\Phi$ (see Remark~\ref{rem:maps:definition:action}).

{The collection of linear maps $\Phi_{k,n}$ defines a coderivation of $\bar S(E\oplus V)$, $$\Upsilon:\bar S(E\oplus V)\to \big(\bar S(E)\otimes \bar S(V)\big)\oplus \bar S(V)\subset \bar S(E\oplus V)$$  related to the 
action $\Phi$ by
$$
\Upsilon(x\otimes v)=\Phi_x(v), \quad x\in E,\, v\in \bar S(V)$$
and $$
\Upsilon(x\otimes v)=\Phi_x(v) + (-1)^{|x_{(1)}|}x_{(1)}\otimes\Phi_{x_{(2)}}(v), \quad x\in S^{\geq 2}(E),\, v\in \bar S(V).
$$}
The degree $+1$ coderivation of $\bar S(E\oplus V)$ determined by $\{\mathfrak{l}_n\}_{n \in \Nn}$ is  $$M_{E\oplus V}=M_E+\Upsilon+M_V.$$ 
Let us prove that $M_{E\oplus V}^2=0$. For $ x \in \bar S(E)$ and $v \in \bar S(V)$,
$$M_{E\oplus V}^2(x)=M_E^2(x)=0 \,\,\,\,\textrm{and}\,\,\,\,  M_{E\oplus V}^2(v)=M_V^2(v)=0$$
while, for mixed terms, we have  $$M_{E\oplus V}(x\otimes v)= M_E(x)\otimes v+(-1)^{|x|}x \otimes M_V(v)+(-1)^{|x_{(1)}|}x_{(1)}\otimes \Phi_{x_{(2)}}(v)+\Phi_x(v)$$
and 
$$  \mathfrak{l}(M_{E\oplus V}(x\otimes v))=(\Phi_{M_E(x)})_{\bullet}\,  (v)+ (-1)^{|x|}(\Phi_x)_{\bullet} \, (M_V(v))+(-1)^{|x_{(1)}|}(\Phi_{x_{(1)}})_{\bullet}( \Phi_{x_{(2)}}(v))+m_{\bullet}(\Phi_x(v)).$$

Since $\Phi$ is a Lie $\infty$-morphism, we have $$\Phi_{M_E(x)}=-M_V\smalcirc \Phi_x-(-1)^{|x|}\Phi_x \smalcirc M_V+ \frac12\brr{\Phi_{x_{(1)}},\Phi_{x_{(2)}}},$$
which  implies $ M_{E\oplus V}^2=0$.
\end{proof}

The Lie $\infty$-algebra structure in $E\oplus V$ presented in Remark~\ref{rem:lie:infty:algebra:E+V}, is a particular case of Proposition~\ref{prop:lie:infty:algebra:E+V}, with $M_V=\d$.

\

\paragraph{\textbf{Adjoint representation and adjoint action}}
An important example of a representation is given by a Lie $\infty$-algebra structure.

Let $\left( E, M_E \equiv \set{l_k}_{k\in\Nn} \right)$ be a Lie $\infty$-algebra; thus $(E,l_1)$ is a complex.
The collection of  degree $+1$ maps
\begin{equation*} \label{eq:adjoint:representation:algebra}
\begin{array}{rrcl}
\ad_k:& S^k(E) &\to&  \End(E) \\
 & \;x_1\odot\ldots \odot x_k & \mapsto &  \ad_{x_1 \odot \dots \odot x_k} := l_{k+1}\left( x_1,\ldots, x_k, \, \, \, -\,\,  \right) 
\end{array}, \quad k\geq 1,
\end{equation*}
 satisfies Equations \eqref{eq:def:representation}. {(Note that Equations \eqref{eq:def:representation} are equivalent to Equations \eqref{eq:def:symm:L:infty:algebra})}. So, this collection of maps defines a representation $\ad=\sum_k \ad_ k$ of the Lie $\infty$-algebra $E$ on $(E,l_1)$. 

\begin{defn}
The representation $\ad$  is called   the \textbf{adjoint representation} of the Lie $\infty$-algebra  $\left( E, M_E \equiv \set{l_k}_{k\in\Nn} \right)$. 
\end{defn}

Moreover, notice that for each $x\in S^i(E)$, {$i\geq 1$}, we may consider the  degree $|x|+1$ coderivation $\ad^D_x$  of $\bar S(E)$ defined by the family of linear maps
\begin{equation*}
\begin{array}{rrcl}
{(\ad_x)_k}:& S^k(E) &\to& E \\
 & e & \mapsto &  l_{i+k}(x,e), \quad {k\geq 1}.
\end{array}
\end{equation*}
So, we have a collection of degree $+1$ linear maps
\begin{equation} \label{eq:coderivation_ad}
\begin{array}{rrcl}
{\ad_i}:& S^i(E) &\to& \Coder(\bar S(E)) \\
 & x & \mapsto & \ad^D_x
\end{array}, \quad {i\geq 1,}
\end{equation}
and we set ${\bf ad}=\sum_i \ad_i$.

\begin{prop}
{The collection of degree $+1$ linear maps given by \eqref{eq:coderivation_ad} defines  a Lie $\infty$-morphism
$${\bf ad}: (E, \set{l_k}_{k\in\Nn}) \to \left(\Coder (\bar S(E))[1], \partial_{M_E}, \brr{\cdot , \cdot} \right)$$
}
from the Lie $\infty$-algebra $E$ to the symmetric DGLA $\Coder (\bar S(E))[1]$.
\end{prop}

\begin{proof}
For each $x\in S^i(E)$,
  let $\ad_x={\sum_k(\ad_x)_k}$ and set $l=\sum_k l_k$.

If $x\in \oplus_{i\geq 2}S^i(E)$  {and $e \in \bar S(E)$, we have
\begin{align*}
M_E(x\odot e)=& M_E(x) \odot e+ (-1)^{|x|}x \odot M_E(e)+ (-1)^{|e||x_{(2)}|}l(x_{(1)}, e)\odot x_{(2)} \\+
& l (x, e_{(1)})\odot e_{(2)}+ (-1)^{|e_{(1)}||x_{(2)}|}l(x_{(1)}, e_{(1)})\odot x_{(2)}\odot e_{(2)} + l(x, e)
\end{align*}
and so,
\begin{align*}
\ad_x (M_E(e))&=l(x,M_E(e))\\
&= (-1)^{|x|}\underbrace{l(M_E(x\odot e))}_{=0 \,\, \textrm{by} \,\, \eqref{eq:def:symm:L:infty:algebra}}- (-1)^{|x|} l (M_E(x), e)- (-1)^{|x|}l(\ad_x^D(e))\\
& \quad - (-1)^{|x_{(1)}|+|x_{(1)}||x_{(2)}| }l(x_{(2)},\ad_{x_{(1)}}^D(e))\\
&= \big(- (-1)^{|x|}\ad_{M_E(x)}- (-1)^{|x|}l \smalcirc \ad^D_{x}-(-1)^{|x_{(2)}|}\ad_{x_{(1)}}\smalcirc \ad^D_{x_{(2)}}\big)(e),
\end{align*}
which is equivalent to
\begin{equation*}
\ad_{M_E(x)}=-l\smalcirc \ad^D_x-(-1)^{|x|}\ad_x \smalcirc M_E - (-1)^{|x_{(2)}|}\ad_{x_{(1)}}\smalcirc \ad^D_{x_{(2)}}
\end{equation*}
or to
\begin{equation} \label{eq:ad:ME}
\ad_{M_E(x)}=-[l, \ad_x]_{_{RN}}-\frac12 (-1)^{|x_{(1)}|}[\ad_{x_{(1)}}, \ad_{x_{(2)}}]_{_{RN}}.
\end{equation}
Note that the coderivation defined by the second member of \eqref{eq:ad:ME} is
\begin{equation*}
[M_E, \ad^D_x]+\frac12 [\ad^D_{x_{(1)}}, \ad^D_{x_{(2)}}]
= \partial_{M_E}(\ad^D_x) + \frac12 [\ad^D_{x_{(1)}}, \ad^D_{x_{(2)}}].
\end{equation*}

If $x\in E$, a similar computation gives
\begin{equation} \label{eq:ad:l1}
\ad_{l_1(x)}=-l \smalcirc \ad^{D}_{x} - (-1)^{|x|}\ad_x\smalcirc M_E = -[l, \ad_x]_{_{RN}}.
\end{equation}
}

Equations \eqref{eq:ad:ME} and \eqref{eq:ad:l1} mean that the map {${\bf ad}: E \to \Coder (\bar S(E))[1]$} is a Lie $\infty$-morphism.
\end{proof}

\begin{defn} The linear map {${\bf ad}: E \to \Coder(\bar S(E))[1]$} is an action of the Lie $\infty$-algebra $E$ on itself, called the \textbf{adjoint action of $E$}.
\end{defn}


\section{\texorpdfstring{$\O$}--operators on a Lie \texorpdfstring{$\infty$}--algebra} \label{section3}

In this section we define $\O$-operators on a Lie $\infty$-algebra $E$ with respect to an action of $E$ on a Lie $\infty$-algebra $V$. This is the main notion of the paper.

\subsection{\texorpdfstring{$\O$}--operators with respect to a Lie \texorpdfstring{$\infty$}--action}

{Let $(E,M_E\equiv\set{l_k}_{k\geq 1})$ and $(V,M_V\equiv\set{m_k}_{k\geq 1})$ be Lie $\infty$-algebras and $\Phi:E\to \Coder(\bar S(V))[1]$ a Lie $\infty$-action of $E$ on $V$.
 Remember we are using Sweedler's notation: for each $v\in \bar S(V)$,
$$\Delta(v)=v_{(1)}\otimes v_{(2)}$$ and
$$\Delta^{2}(v)=(\mathrm{id} \otimes \Delta)\Delta(v)=( \Delta\otimes\mathrm{id})\Delta(v)=v_{(1)}\otimes v_{(2)}\otimes v_{(3)}.$$

Each degree zero linear map $T:\bar S(V)\to \bar S(E)$ defines  a degree $+1$ linear map $\ds \Phi^T: \bar S(V)\to \bar S(V)$ given by
\begin{align*}
   \Phi^T(v)&=0,\quad v\in V,\\
\Phi^T(v)&=\Phi_{T(v_{(1)})}\, v_{(2)},\quad v\in S^{\geq 2}(V).
\end{align*}

}

\begin{lem}\label{Lemma:deformed:coderivation:by:T}
The linear map $\ds \Phi^T: \bar S(V)\to \bar S(V)$ is a degree $+1$ coderivation of $\bar S(V)$ and is defined by the collection of linear maps
$\sum\Phi_{\bullet,\bullet}(T\otimes\mathrm{
id})\Delta$.
\end{lem}

\begin{proof}
For the linear map $\ds \Phi^T: \bar S(V)\to \bar S(V)$ to be a coderivation  it must  satisfy:
$$
\Delta\Phi^T(v)=\left(\Phi^T\otimes \mathrm{id} + \mathrm{id}\otimes \Phi^T\right)\Delta(v),\quad v\in \bar S(V).
$$

This equation is trivially satisfied for  $v\in V$.

For each $v=v_1\odot v_2\in S^2(V)$ we have
$\Phi^T(v)\in V$ and consequently, $\Delta\Phi^T(v)=0$. On the other hand, since ${\Phi^T}_{|V}=0$, we see that
$$\left(\Phi^T\otimes \mathrm{id} + \mathrm{id}\otimes \Phi^T\right)\Delta(v)=0$$
and the equation is satisfied in $S^2(V)$.

Now let  $v\in S^{\geq 3}(V)$, then
\begin{align*}
\Delta\Phi^T(v)&=\Delta \Phi_{T(v_{(1)})}v_{(2)}\\
&= \left(\Phi_{T(v_{(1)})}\otimes\mathrm{id} + \mathrm{id}\otimes \Phi_{T(v_{(1)})}\right)\Delta (v_{(2)}).
\end{align*}
The coassociativity of $\Delta$ ensures that 
\begin{align*}
\Delta\Phi^T(v)&= \Phi_{T(v_{(1)})}v_{(2)}\otimes v_{(3)} + (-1)^{(|v_{(1)}|+1)|v_{(2)}|}v_{(2)}\otimes \Phi_{T(v_{(1)})}v_{(3)}\\
&=\Phi_{T(v_{(1)})}v_{(2)}\otimes v_{(3)} + (-1)^{|v_{(1)}|} v_{(1)}\otimes \Phi_{T(v_{(2)})}v_{(3)}\\
&= \left(\Phi^T\otimes \mathrm{id} + \mathrm{id}\otimes \Phi^T\right)\Delta(v).
\end{align*}
\end{proof}

\begin{defn}
Let $(E, M_E \equiv \set{l_k}_{k\geq 1})$ and $(V, M_V \equiv \set{m_k}_{k\geq 1})$ be
Lie $\infty$-algebras and $\Phi:E\to \Coder (\bar S(V))[1]$ an action.
An \textbf{$\O$-operator} {on $E$ with respect to the action $\Phi$} is  a (degree $0$) morphism of coalgebras  $T:\bar S(V) \to \bar S(E)$ such that
\begin{equation}\label{def:O:operator}
{M_E\smalcirc T=T\smalcirc\left(\Phi^T+M_V\right).}
\end{equation}
\end{defn}

\begin{defn}  A \textbf{Rota Baxter operator (of weight $1$)} on a Lie $\infty$-algebra $( E, M_E \equiv \set{l_k}_{k\geq 1})$ is an $\O$-operator  with respect to the adjoint action.

\end{defn}

An $\O$-operator $T:\bar S(V)\to \bar S(E)$ with respect to an action $\Phi:E\to \Coder(\bar S(V))[1]$ of $(E,M_E\equiv \set{l_k}_{k\geq 1})$ on $(V,M_V\equiv\set{m_k}_{k\geq 1})$ is defined by a linear map ${t=\sum_{i}t_{i}}:\bar S(V)\to E$
satisfying:
\begin{itemize}
  \item[(i)] $\ds l_1( t_1(v))= t_1(m_1(v)), \quad v\in V$
  \item[(ii)] $l(T(v))=t\left(\Phi_{T(v_{(1)})}v_{(2)} +  m(v_{(1)})\odot v_{(2)}\right), \quad v\in \oplus_{i\geq 2}S^i(V).$
\end{itemize}
In particular, the $\O$-operator $T$ is a comorphism i.e., for each  $v \in  {S^n(V)}$, $n\geq 1$,
\begin{equation*}
T(v)= \sum_{k_1+ \ldots +k_r=n}\frac{1}{r!}\, t_{k_1}(v_{(k_{1})})\odot \ldots \odot t_{k_r}(v_{(k_r)}),
\end{equation*}
so, detailing (ii)  for $v=v_1\odot v_2$, we get

\begin{align*}
    l_1\Big(t_2(v_1, \,v_2)\Big) + l_2\Big(t_1(v_1),\, t_1(v_2)\Big) &=
    \,t_1\left(\Phi_{t_1(v_1)}v_2 + (-1)^{|v_1|\,|v_2|}\Phi_{t_1(v_2)}v_1 + m_2(v_1,v_2)\right)\\
    &+ t_2\Big(m_1(v_1), v_2)\Big) + (-1)^{|v_1|}t_2\Big(v_1,m_1(v_2)\Big).
\end{align*}
{Generally, for every $v=v_1\odot \ldots \odot v_n\in S^n(V)$, $n\geq 3$,  we have

{\small
\begin{eqnarray*} \label{eq:expression:T}
\lefteqn{\sum_{\begin{array}{c} \scriptstyle{k_1+\ldots +k_i=n}\\ \scriptstyle{\sigma \in Sh(k_1, \ldots, k_i)}\end{array}}
\frac{\epsilon(\sigma)}{i!}\, l_i\bigg(t_{k_{1}}(v_{\sigma(1)}, \dots, v_{\sigma(k_1)}), \ldots, t_{k_{i}}(v_{\sigma(k_1+ \ldots + k_{i-1}+1)}, \dots, v_{\sigma({n})})\bigg) }\\ 
 &=&\!\!\!\!\!\!\!\!\!\!\sum_{\begin{array}{c} \scriptstyle{k_1+\ldots +k_{i+2}=n}\\ 
\scriptstyle{\sigma \in Sh(k_1, \ldots, k_{i+2})}\end{array}}
 \!\!\!\!\!\!\!\!\!\!\frac{\epsilon(\sigma)}{i!}\, t_{1+k_{i+2}}\bigg(
 \Phi_{i,{k_{i+1}}}
 \bigg( 
 t_{k_{1}}(v_{\sigma(1)} \dots, v_{\sigma(k_1)})\odot \ldots \odot t_{k_{i}}(v_{\sigma(k_1+ \ldots + k_{i-1}+1)}, \dots, 
 v_{\sigma({ k_1+\ldots + k_i})}), \nonumber \\
 & &\qquad v_{\sigma(k_1+ \ldots + k_{i}+1)}
 \odot \dots\odot v_{\sigma({k_1+\ldots + k_{i+1}})}  \bigg),    v_{\sigma(k_1+ \ldots + k_{i+1}+1)}\odot \dots \odot v_{\sigma (n)}
 \bigg)  \\
 &&+ {\sum_{i=1}^{n}}\,\sum_{\sigma \in Sh(i,n-i)} \epsilon(\sigma)\, {t_{n-i+1}}\, \big(m_i(v_{\sigma(1)}, \dots, v_{\sigma(i)}), v_{\sigma(i+1)}, \ldots, v_{\sigma(n)}\big). \nonumber
\end{eqnarray*}}}

\begin{rem}
When   $M_V=0$
we are considering $V$ simply as a graded vector space, with no Lie $\infty$-algebra attached and an $\O$-operator must satisfy
$$
{M_E\smalcirc T=T\smalcirc\Phi^T.}
$$
In this case, the terms of {above equations} 
involving the brackets $m_i$ on $V$ vanish.
\end{rem}

\begin{rem}
When  $(E, \brr{\cdot,\cdot}_E)$ and $(V, \brr{\cdot,\cdot}_V)$ are  Lie  algebras, for degree reasons, a morphism $T=t_{1}$ must be a strict morphism.  Moreover, our definition coincides with the usual definition of $\O$-operator (of weight $1$) between Lie algebras  \cite{K} :

$$
\brr{t_{1}(v),t_{1}(w)}_{E}=t_{1}\left(\Phi_{t_{1}(v)}w - \Phi_{t_{1}(w)}v+\brr{v,w}_{V}\right), \quad v,w\in V.
$$
\end{rem}

\begin{rem}

When $(V, \d)$ is just a complex and the action $\Phi:E\to\Coder (\bar S(V)) [1]$ is induced by a  representation
$\rho:E\to \End(V)[1]$ we have that $\Phi(x)$ is the {(co)derivation}  defined by $\rho(x)$. In this case,  $\O$-operators with respect to $\Phi$ coincide with $\O$-operators with respect to $\rho$ 
 (or relative Rota Baxter operators) given in 
 \cite{LST2021}.
 \end{rem}

In \cite{LST2021} the authors define $\O$-operators with respect to   representations of Lie $\infty$-algebras. Any action   induces a representation and  $\O$-operators with respect to  an action are related with  $\O$-operators of with respect to the induced representation. We shall see that this relation is given by the comorphism $$I=\sum_{n\geq 1} {\mathrm{i}}_n:\bar S(V)\to \bar S(\bar S(V)),$$ defined by
the family of inclusion maps
$\ds \mathrm{i}_n:S^n(V)\hookrightarrow \bar{S}(V)$, $n\geq 1$. 

 Notice that any coderivation $D$ of $\bar S(V)$ induces a (co)derivation $D^{\d}$ of $\bar S(\bar S(V))$. The comorphism $I$ preserves these coderivations:
 
\begin{lem}\label{lemma:I:Lie:morphism}
Let $V$  be a graded vector space and $D$ a coderivation of $\bar S(V)$. 
The map $I:\bar S(V)\to \bar S(\bar S(V)) $ satisfies
$$
I\smalcirc D=D^{\d}\smalcirc I.
$$
\end{lem}

\begin{proof}
We will denote  by $\cdot$ the symmetric product in $\bar S(\bar S(V))$, to distinguish from the symmetric product $\odot$ in $\bar S(V)$. 

Let $v\in  S^n(V)$, $n\geq 1$, and denote by   
$\set{m_k}_{k\geq 1}$ the family of linear maps defining the coderivation $D$.
For  $v\in V$, we immediately have   $D^{\d}\smalcirc I(v)= D\smalcirc I(v)=I\smalcirc D(v)$.
For $v\in S^{\geq 2}(V)$ we have
\begin{eqnarray*}
   \lefteqn{ D^{\d}\smalcirc I(v)= D^{\d}\bigg(\sum_{k=1}^n \frac{1}{k!}v_{(1)}\cdot \ldots\cdot v_{(k)}\bigg)}\\
    &=&\sum_{k=1}^n \frac{1}{k!}\bigg(D(v_{(1)})\cdot v_{(2)} \cdot \ldots\cdot v_{(k)}+\ldots + (-1)^{|D|(|v_{(1)}|+\ldots+|v_{(k-1)}|)}v_{(1)} \cdot \ldots\cdot v_{(k-1)} \cdot D(v_{(k)})\bigg)\\
    &=&\sum_{k=1}^n \frac{1}{(k-1)!}D(v_{(1)})\cdot v_{(2)} \cdot \ldots\cdot v_{(k)}\\
    &=&D(v_{(1)})  \cdot I(v_{(2)}).
\end{eqnarray*}
On the other hand 
\begin{align*}
    I\smalcirc D(v) &=I(m_{\bullet}(v_{(1)})\odot v_{(2)})\\
    &=m_{\bullet}(v_{(1)})\cdot I(v_{(2)}) + (m_{\bullet}(v_{(1)})\odot v_{(2)})\cdot I(v_{(3)})  \\
   &=D(v_{(1)})  \cdot I(v_{(2)})
\end{align*}
and the result follows.
\end{proof}

\begin{rem}
In particular, if $D$ defines a Lie $\infty$-algebra structure on $V$, then $D^{\d}$ defines a Lie $\infty$-algebra structure on $\bar S(V)$ and $I$ is a Lie $\infty$-morphism.  
\end{rem}

\begin{prop}
Let $\Phi:E\to \Coder (\bar S(V))[1]$ be an action of the Lie $\infty$-algebra $\left( E, M_E \equiv \set{l_k}_{k\geq 1} \right)$ on the Lie $\infty$-algebra $\left(V, M_V \equiv \set{m_k}_{k\geq 1} \right)$ and $\tilde T:\bar S(\bar S(V))\to E$ be an $\O$-operator  with respect to the induced representation $\rho:E\to \End(\bar S(V))[1]$. Then
$T=\tilde T \smalcirc I$ is an $\O$-operator with respect to the action $\Phi$.
\end{prop}

\begin{proof}
For each  $x\in \bar S(E)$, let us denote by
$$\Phi_x^{\d}:=\Phi(x)^{\d}=\rho(x)^{\d},
$$ the (co)derivation of $\bar S(\bar S(V))$ defined by $\rho(x)$. 

Let $\tilde T$ be an $\O$-operator with respect to the induced representation.
This means that
$$
M_E\smalcirc \tilde T(w)=\tilde T\Big(\Phi_{T(w_{(1)})}^{\d} w_{(2)} + {M_V}^{\d}(w)\Big), \quad w\in \bar S(\bar S(V)).
$$
Then,   
for each $w=I(v)$, $v\in\bar S(V)$, we have:
\begin{align*}
    M_E\smalcirc \tilde T(I(v))&= 
    \tilde T\Big(\Phi_{\tilde T(I(v)_{(1)})}^{\d} I(v)_{(2)} + {M_V}^{\d}\smalcirc I(v)\Big).
\end{align*}
Using  the fact that $I$ is a comorphism and Lemma \ref{lemma:I:Lie:morphism}, we rewrite last equation as
\begin{align*}
M_E\smalcirc T(v)&=
 \tilde T\Big(\Phi_{\tilde T(I(v_{(1)}))}^{\d} I(v_{(2)}) + {M_V}^{\d}\smalcirc I(v)\Big)\\
&=\tilde T \Big(I\smalcirc\Phi_{T(v_{(1)})} v_{(2)} + I\smalcirc {M_V}(v)\Big) \\
&=T \Big(\Phi_{T(v_{(1)})} v_{(2)} + {M_V}(v)\Big).
\end{align*}
Taking into account this equation and that $T$ is a comorphism, because is the composition of two comorphisms,  the result follows.
\end{proof}

\begin{prop}
Let $T$ be an $\O$-operator on $\left( E, M_E \equiv \set{l_k}_{k\geq 1} \right)$ with respect to a Lie $\infty$-action $\Phi:E\to \Coder (\bar S(V))[1]$ on $\left( V, M_V \equiv \set{m_k}_{k\geq 1} \right)$. Then, $V$ has a new Lie $\infty$-algebra structure {$$M_{V^{T}}= \Phi^T + M_V$$} and $T:(V, M_{V^{T}})\to (E, M_E)$  is a Lie $\infty$-morphism.
\end{prop}

\begin{proof}
By Lemma \ref{Lemma:deformed:coderivation:by:T} we know $\Phi^T$ is a degree $+1$ coderivation of $\bar S(V)$ hence so is $M_{V^T}$.

Since $\Phi$ is an action, {so that $\Phi \smalcirc M_E= M_{\textrm{Coder}(\bar S(V))[1]} \smalcirc \Phi$,} and $T$ is a comorphism, we have, for each $v\in \bar S(V)$,
\begin{eqnarray}
\label{equation:first:O:operators}
\Phi_{M_{E}T(v_{(1)})}v_{(2)}&=& -M_{V} \Phi_{T(v_{(1)})}v_{(2)} - (-1)^{|v_{(1)}|} \Phi_{T(v_{(1)})} M_{V}( v_{(2)}) \\
&&\quad + (-1)^{|v_{(1)}|+1} \Phi_{T(v_{(1)})}\Phi_{T(v_{(2)})} v_{(3)}.\nonumber
\end{eqnarray}

On the other hand, $T$ is an $\O$-operator: $$M_{E}\smalcirc T(v)=T\smalcirc \Phi_{T(v_{(1)})}v_{(2)} + T\smalcirc M_{V}(v)$$ and this yields
\begin{eqnarray}\label{equation:second:O:operators}
\Phi_{M_{E}T(v_{(1)})}v_{(2)}=\Phi_{T\Phi_{T(v_{(1)})}v_{(2)} } v_{(3)} + \Phi_{TM_{V}(v_{(1)})}v_{(2)}.
\end{eqnarray}

Moreover,
due to the fact that both $\Phi^T$ and $M_V$ are coderivations and $M_V^2=0$, we have
\begin{align*}
M_{V^{T}}^{2} (v)&= (\Phi^T)^2(v)+\Phi^T\smalcirc M_V(v) + M_V\smalcirc \Phi^T(v)\\
&= \Phi_{T(\Phi_{T(v_{(1)})}v_{(2)})} v_{(3)} + (-1)^{|v_{(1)}|} \Phi_{T(v_{(1)})}\Phi_{T(v_{(2)})} v_{(3)} \\
&\quad +
\Phi_{TM_{V}(v_{(1)})}v_{(2)} +  (-1)^{|v_{(1)}|}
\Phi_{T(v_{(1))}}M_{V}(v_{(2)} ) + M_{V}(\Phi_{T(v_{(1)})}v_{(2)}).
\end{align*}
Taking into account Equations (\ref{equation:first:O:operators}) and (\ref{equation:second:O:operators}) we conclude $\ds M_{V^T}^2=0$. Therefore, 
$M_{V^{T}}$ defines a Lie $\infty$-algebra structure on $V$ and Equation (\ref{def:O:operator}) means that $T: \bar S(V) \to \bar S(E)$ is a Lie $\infty$-morphism between the Lie $\infty$-algebras $(V, M_{V^{T}})$ and $(E, M_E)$.
\end{proof}

The  brackets of the Lie $\infty$-algebra structure on $V$ defined by the coderivation $M_{V^{T}}$ are given by
$$m_1^T(v)= m_1(v)$$
and, for $n\geq 2$,
\begin{align*}
&m_n^T(v_1, \ldots, v_n)= m_n(v_1, \ldots v_n)+ \, \sum_{\begin{array}{c} \scriptstyle{k_1+\ldots +k_i=j}\\\scriptstyle{1\leq j \leq n-1}\end{array}}\sum_{\sigma \in Sh(k_1, \ldots, k_i,n-j)} \epsilon(\sigma)\, \frac1{n!}\\
&\Phi_{i,n-{j}}\left( t_{k_1}(v_{\sigma(1)}, \ldots, v_{\sigma(k_1)})\odot \ldots\odot t_{k_i}(v_{\sigma(k_1 + \dots + k_{i-1}+1)}, \dots , v_{\sigma(j)}), v_{\sigma(j+1)}\odot \dots\odot v_{\sigma(n)}  \right),
\end{align*}
with $\Phi_{i,n-{j}}\, , i\geq 1,$ the linear maps determined by the action $\Phi$ (see  \eqref{eq:linear:maps:representation}).

\

\paragraph{\textbf{$\O$-operators for the coadjoint representation}}
Let $(E, M_E\equiv\set{ l_k}_{k\geq 1})$ be a finite dimensional Lie $\infty$-algebra. 
Next, we consider the dual of the adjoint representation of $E$ (see \eqref{eq:dual:representation}), called the coadjoint representation.

\begin{defn}
The \textbf{coadjoint representation} of $E$,
$\ds
\ad^{*}:E \to \End (E^{*})[1]
$, is defined by 
\begin{equation*}
\eval{\ad^{*}_{x}(\al),v}=-(-1)^{| \al |(|x|+1)}\eval{\al, \ad_{x}v},\quad v\in E,\, x\in \bar S(E),\, \al\in E^{*}.
\end{equation*}
\end{defn}
Notice that $E^*$ is equipped with the differential $l_{1}^{*}$ (see \eqref{eq:dual:map}). 

\

An $\O$-operator on $E$  with respect to the coadjoint representation $\ad^*:E\to \End{(E^*)}[1]$ is a coalgebra morphism $T:\bar S (E^{*})\to \bar S(E)$  given by a collection of maps $t=\sum_{i} t_{i}:\bar S(E^{*})\to E$ satisfying
{
\begin{eqnarray}\label{equation:Ooperator:coadjoint}
l(T(\al))&=&\!\!\!\!\!\!\!\!\!\!\!\!\sum_{\begin{array}{c} \scriptstyle{1\leq i\leq n-1}\\\scriptstyle{\sigma\in Sh(i,n-i)}\end{array}}\!\!\!\!\!\! \varepsilon(\sigma)\,t_{n-i+1}(\ad^{*}_{T( \al_{\sigma(1)}\odot \ldots\odot\al_{\sigma(i)})} \al_{\sigma(i+1)}, \al_{\sigma(i+2)}, \ldots,
\al_{\sigma(n)})  \nonumber \\
&&+ \sum_{{i=1}}^{n} (-1)^{|\al_{1}| + \dots + |\al_{i-1}|}t_n(\al_{1},\ldots, l_{1}^{*}\al_{i}, \ldots,\al_{n}),
\end{eqnarray}
}
for all $\al=\al_{1}\odot\ldots\odot \al_{n} \in  S^{n}(E^{*})$, $n\geq 1$. 

We say that $T$ is {\bf symmetric} if $$\eval{\beta, t_{n}(\al_{1},\ldots,\al_{n})}=(-1)^{|\al||\beta|+|\al_{n}|(|\al_{1}|+\ldots +|\al_{n-1}|)}\eval{\al_{n}, t_{n}(\al_{1},\ldots,\al_{n-1},\beta)}, $$ for all
$\al_1,\ldots, \al_n, \beta\in E^{*}$ and $n\geq 1$.

When $T$ is invertible, its inverse $T^{-1}:\bar S (E)\to \bar S(E^{*})$, given by $t^{-1}=\sum_{n}t^{-1}_{n}$, is also symmetric:
$$\eval{ t^{-1}_{n}(x_{1},\ldots, x_{n}), y}=(-1)^{|y||x_{n}|}\eval{t^{-1}_{n}(x_{1},\ldots, x_{n-1},y), x_{n}},$$
for every $ x_1,\ldots x_n, y\in E$, $n\geq 1$.

One should notice that $t^{-1}_{n}$ is \textbf{not} the inverse map of $t_{n}$. It simply denotes the $n$-component of the inverse $T^{-1}$ of $T$.

For each $n\geq 1$, let 
$\omega^{(n)}\in \otimes^n E^*$ be defined by $\omega^{(1)}=0$ and
$$\eval{\omega^{(n)}, x_{1}\otimes \ldots \otimes x_{n}}=\eval{ t^{-1}_{n-1}(x_{1}, \ldots, x_{n-1}), x_{n}}, \quad x_1,\dots, x_n\in E.$$

 The symmetry of $T^{-1}$ guarantees that
$\omega=\sum_{n\geq 1} \omega^{(n)}$ is an element of
 $\bar S(E^{*})$. 

\begin{prop}
Let $T:\bar S(E^{*})\to \bar S(E)$  be an invertible symmetric comorphism. The linear map $T$ is an $\O$-operator with respect to the coadjoint representation  if and only if $\omega\in \oplus_{n\geq 2} S^{n}(E^{*})$, given by
$$
\eval{\omega, x_{1} \odot\ldots\odot x_{k+1}}=\eval{t_{k}^{-1}(x_{1},\ldots, x_{k}), x_{k+1}}, \quad x_{1}, \ldots, x_{k+1}\in E,\, k\geq 1,
$$
is a cocycle for the Lie $\infty$-algebra  cohomology.
\end{prop}

\begin{proof}
When $T$ is invertible, Equation  (\ref{equation:Ooperator:coadjoint}) is equivalent to equations
$$t_1^{-1}l_{1}(x)=l_{1}^{*}t_1^{-1}(x), \quad x\in E,$$
and
{
$$t^{-1}M_{E}(x)=\ad^{*}_{x_{(1)}} t^{-1}(x_{(2)})  + l_{1}^{*}t_n^{-1}(x),\quad x\in S^{n}(E), n\geq 2.$$
}

Let $x=x_1\odot\ldots\odot x_n\in   S^n(E)$, $n\geq 1$, and $y\in E$, such that $|y|=|x|+1$. We have:
{\begin{align*}
\eval{\omega, M_{E}(x\odot y)} &= \eval{ t^{-1}(M_{E} (x)),y} +(-1)^{|x|} \eval{t^{-1}(x_{1},\ldots, x_{n}), l_{1}(y) }\\
&\quad + 
(-1)^{|x_{(1)}|} \eval{t^{-1}(x_{(1)}), \ad_{x_{(2)}} y } \\
&=\eval{ t^{-1}(M_{E} (x)),y} -\eval{l_{1}^{*}t^{-1}(x), y }  -
 \eval{\ad^*_{x_{(1)}}t^{-1}(x_{(2)}),  y }\\
\end{align*}
}
and the result follows.
\end{proof}

\subsection{\texorpdfstring{$\O$}--operators as Maurer-Cartan elements}

Let $\ds (E,M_E\equiv\set{l_k}_{k\geq 1})$ and \linebreak
$\ds (V, M_V\equiv\set{m_k}_{k\geq 1})$ be Lie $\infty$-algebras.

The graded vector space of linear maps between $\bar S(V)$ and $E$ will be denoted by $\mathfrak{h}:= {\mathrm{Hom}} (\bar S(V), E)$.
It can be identified  with the space of  coalgebra morphisms between $\bar S(V)$ and $\bar S(E)$.
On the other hand, since
$$S^n(E\oplus V)\simeq \oplus_{k=0}^n \left({ S^{n-k}(E)\otimes S^{k}(V)}\right),\quad  n\geq 1,$$ the space $\mathfrak{h}$ can be seen as a subspace of $\Coder (\bar S(E\oplus V))$, the space of coderivations of
$\bar S(E\oplus V)$. Its elements define  coderivations that only act on elements of $\bar S(V)$,  they are $S(E)$-linear.

The space $S(E\oplus V)$ has a natural  $S(E)$-bimodule structure. With the above identification we have:
$$
e\cdot (x\otimes v)=(e\odot x)\otimes v=(-1)^{|e|(|x|+|v|)}(x\otimes v)\cdot e,
$$
for $e\in S(E)$, $x\otimes v\in S(E\oplus V)\simeq S(E)\otimes S(V)$.

Let $t:\bar S(V)\to E$ be an element of $\mathfrak{h}$ defined by the collection of maps
$t_k: S^{k}(V)\to E$, $k\geq 1$. Let us denote by  $T: \bar S(V)\to \bar S(E)$  the coalgebra morphism and by $\mathfrak{t}$  the coderivation of $\bar S(E\oplus V)$ defined by $t$. 
 Notice that $$\mathfrak{t}(v)=t_1(v), \quad v\in V$$ 
and
$$ \mathfrak{t}(v)=t(v_{(1)})\otimes v_{(2)} + t(v), \quad v\in S^{\geq 2}(V).$$
and also, for $x \in \bar S (E)$, $$\mathfrak{t}(x \otimes v)=(-1)^{|x||t|}x \cdot \mathfrak{t}(v), \quad v \in \bar S(V).$$

\begin{prop} \label{prop:h:abelian:subalgebra}
The  space $\mathfrak{h}$ is an abelian Lie subalgebra of $\Coder (\bar S(E \oplus V))$.
\end{prop}

\begin{proof}
Let $t=\sum_{i} t_{i} :\bar S(V)\to E$ and $w=\sum_{i} w_{i} :\bar S(V)\to E$ be  elements of $\mathfrak{h}$. Denote by
 $\mathfrak{t}$ and $\mathfrak{w}$
 the coderivations of $\bar S(E\oplus V)$ defined by $t$ and $w$, respectively.

{Let  $v \in \bar S(V)$.} The Lie bracket of $\mathfrak{t}$ and $\mathfrak{w}$ is given by:
 \begin{align*}
 \brr{\mathfrak{t}, \mathfrak{w}}_{c}(v)&=\mathfrak{t}\smalcirc ({w}(v_{(1)})\otimes v_{(2)})-(-1)^{|t||w|}\mathfrak{w}\smalcirc ({t}(v_{(1)})\otimes v_{(2)})\\
 &=(-1)^{|t|(|w|+|v_{(1)}|)}{w}(v_{(1)})\cdot \mathfrak{t}(v_{(2)})-(-1)^{|t||w|}(-1)^{|w|(|t|+|v_{(1)}|)}{ t}(v_{(1)})\cdot \mathfrak{w}(v_{(2)})\\
 &=\Big((-1)^{|t|(|w|+|v_{(1)}|)}{ w}(v_{(1)})\cdot {t}(v_{(2)})-(-1)^{|t||w|}(-1)^{|w|(|t|+|v_{(1)}|)}{t}(v_{(1)})\cdot {w}(v_{(2)})\Big)\otimes v_{(3)}\\
 & \quad + (-1)^{|t|(|w|+|v_{(1)}|)}{ w}(v_{(1)})\cdot {t}(v_{(2)})-(-1)^{|t||w|}(-1)^{|w|(|t|+|v_{(1)}|)}{t}(v_{(1)})\cdot {w}(v_{(2)})\\
 &=\Big((-1)^{|t|(|w|+|v_{(1)}|)}{w}(v_{(1)})\odot {t}(v_{(2)})-(-1)^{|t|(|w|+|v_{(2)}|)+{|v_{(1)}||v_{(2)}|}}{w}(v_{(2)})\odot {t}(v_{(1)})\Big)\otimes v_{(3)}\\
 &\quad +(-1)^{|t|(|w|+|v_{(1)}|)}{w}(v_{(1)})\odot {t}(v_{(2)})-(-1)^{|t|(|w|+|v_{(2)}|)+{|v_{(1)}||v_{(2)}|}}{w}(v_{(2)})\odot {t}(v_{(1)}),
 \end{align*}
{where we used the fact that $\mathfrak{t}$ and $\mathfrak{w}$ are $\bar S(E)$-linear.}
Because of cocommutativity of the coproduct,   the last expression vanishes.
\end{proof}

Now, let $\Phi:E\to \Coder (\bar S(V))[1]$ be an action of the Lie $\infty$-algebra $E$ on the Lie $\infty$-algebra $V$. {By Proposition~\ref{prop:lie:infty:algebra:E+V}, $\Phi$ induces a coderivation $\Upsilon$ of $\bar S(E\oplus V)$ and $M_{E \oplus V}= M_E + \Upsilon+ M_V$ is a Lie $\infty$-algebra structure on $E\oplus V$. Let $\mathcal{P}:\Coder (\bar S(E\oplus V))\to \mathfrak{h}$ be the projection onto $\mathfrak{h}$.} 

Then we have:

\begin{prop}  \label{prop:Vdata:h}
The quadruple $\ds \left(\Coder (\bar S(E\oplus V)), \mathfrak{h},  \mathcal{P}, M_{E \oplus V}\right)$ is a $V$-data and $\mathfrak{h}$ has a Lie $\infty$-algebra structure.
\end{prop}

\begin{proof}
We already know that  $\Coder (\bar S(E\oplus V))$, equipped with the commutator, is a graded Lie algebra and $\mathfrak{h}$ is an abelian Lie subalgebra.

Let $p:\bar S(E\oplus V)\to E$ be the projection and $i: \bar S(V)\to  \bar S(E\oplus V)$ the inclusion.

Notice that, for each $Q\in\Coder (\bar S(E\oplus V))$ we have $\ds \mathcal{P}(Q)=p\smalcirc Q\smalcirc i$ so
 $$\ker \mathcal{P}=\set{Q\in\Coder (\bar S(E\oplus V)): Q\smalcirc  i \mbox{ is a coderivation of } \bar S(V)}$$ 
 is clearly a Lie subalgebra of $\Coder (\bar S(E\oplus V))$:
 \begin{align*}
\mathcal{P}(\brr{Q,P}_{c})&= p\smalcirc \brr{Q,P}_{c}\smalcirc i=p\smalcirc QP\smalcirc i - (-1)^{|Q||P|}p\smalcirc PQ\smalcirc i\\
&=p\smalcirc Q\smalcirc i\smalcirc P\smalcirc i - (-1)^{|Q||P|}p\smalcirc P\smalcirc i\smalcirc Q\smalcirc i=0, \quad P,Q\in\ker \mathcal{P}.
\end{align*}
Moreover
$$M_{E \oplus V} \smalcirc i =M_{V}, \mbox{ so } M_{E \oplus V}\in (\ker\mathcal{P})_{1}$$
and, since $M_{E \oplus V}$ defines a Lie $\infty$-structure in $E\oplus V$, we have:
$$
\brr{M_{E \oplus V},M_{E \oplus V}}_{c}=0.
$$

Voronov's construction \cite{V05} guarantees that $\mathfrak{h}$ inherits a (symmetric) Lie $\infty$-structure given by:

\begin{align*}
\partial_{k}\big(t_{1},\ldots,t_{k})=\mathcal{P}([[\ldots\brr{M_{E \oplus V}, t_{1}}_{_{RN}}\ldots]_{_{RN}}, t_{k}]_{_{RN}}\big), \quad t_{1},\ldots,t_{k}\in\mathfrak{h}, \, k\geq 1.
\end{align*}
\end{proof}

\begin{rem}
A similar proof as in  \cite{LST2021} shows that, with the above structure, $\mathfrak{h}$ is a filtered Lie $\infty$-algebra.
\end{rem}

\begin{lem}\label{lemma:bracket:voronov}
Let 
$t:\bar S(V)\to E$ be a degree zero element of $\mathfrak{h}$.
For each $v\in \bar S(V)$,
\begin{align*}
\partial_{1}t(v)=l_{1}t(v)-t\smalcirc M_{V}(v)
\end{align*}
and
\begin{align*}
\partial_{k}(t,\ldots,t)(v)&=l_{k}\big(t(v_{(1)}),\ldots, t(v_{(k)})\big)-k \, t \left(\Phi_{t(v_{(1)})
\odot\ldots\odot t(v_{(k-1)})} v_{(k)}\right), \quad k\geq 2.
\end{align*}

\begin{proof}
Let $p:{\bar S(E\oplus V)}\to E$ be the projection map and $\mathfrak{t}$ the  coderivation of $\bar S(E\oplus V)$ defined by $t$.
 Notice that
\begin{align*}
    p\smalcirc \mathfrak{t}&=  t  \\
    p\smalcirc \mathfrak{t}^{k}&=  0, \quad k\geq 2.
\end{align*}
Consequently, for $k=1$  we have
\begin{align*}
\partial_{1}t(v)=p\smalcirc M_{E \oplus V}\smalcirc {\mathfrak{t}}(v)- p\smalcirc \mathfrak{t}\smalcirc M_{V}(v)=l_{1}t(v)-t\smalcirc M_{V}(v),
\,\,{v \in \bar S(V)}
\end{align*}
and, for $k\geq 2$,
\begin{align*}
\partial_{k}(t,\dots,t)&=p\smalcirc M_{E \oplus V}\smalcirc \mathfrak{t}^{k} - k \, p\smalcirc \mathfrak{t}\smalcirc M_{E \oplus V}\smalcirc \mathfrak{t}^{k-1}\\
&=l\smalcirc \mathfrak{t}^{k} - k\,t\smalcirc M_{E \oplus V}\smalcirc \mathfrak{t}^{k-1}
\end{align*}
and the result follows.
\end{proof}
\end{lem}

\begin{rem}
Notice that     
$\ds \partial_{k}(t,\dots,t)(v)=0$, for $v\in S^{<k}(V)$, as a consequence of $\mathfrak{t}^{k}(v)=0$ and $\Phi\smalcirc \mathfrak{t}^{k-1}(v) \in \bar S(E)$.
\end{rem}

Next proposition realizes $\mathcal{O}$-operators as Maurer-Cartan elements of this Lie $\infty$-algebra $\mathfrak{h}$.

\begin{prop} \label{prop:1-1correspondence}
$\O$-operators {on $E$ with respect to an action $\Phi$} are  
Maurer-Cartan elements of $\mathfrak{h}$.
\end{prop}

\begin{proof} 

Let $t:\bar S(V)\to E$ be a degree zero element of $\mathfrak{h}$. 
Maurer-Cartan equation yields
$$
\partial_{1}t + \frac{1}{2}\partial_{2}(t,t) + \dots+ \frac{1}{k!} \partial_{k}(t, \ldots,t) + \ldots =0.
$$
Using Lemma \ref{lemma:bracket:voronov} we have, for each $v\in S^{k}(V)$,
\begin{align*}
  \partial_{1}t (v)+& \frac{1}{2}\partial_{2}(t,t) (v)+ \dots +\frac{1}{k!} \partial_{k}(t, \ldots,t) (v)  =\\
 &= l_{1}t(v)-t\smalcirc M_{V}(v) \\
  & +\frac{1}{2}l_{2}\big(t(v_{(1)}), t(v_{(2)})\big)- t \left(\Phi_{t(v_{(1)})} v_{(2)}\right)+ \ldots +\\
   & +\frac{1}{k!}l_{k}\big(t(v_{(1)}),\ldots, t(v_{(k)})\big)-\frac{1}{(k-1)!} \, t \left(\Phi_{t(v_{(1)})
\odot\ldots\odot t(v_{(k-1)})} v_{(k)}\right).\\
\end{align*}

Let $T:\bar S(V)\to \bar S(E)$  be the morphism of coalgebras defined by $t:\bar S(V)\to E$.
 Maurer-Cartan equation can be written as
$$
l\smalcirc T(v) - t\smalcirc M_{V}(v) - t\,\Phi_{T(v_{(1)})}v_{(2)}=0,
$$
which is equivalent to  $T$ being an $\O$-operator (see Equation (\ref{def:O:operator})).
\end{proof}

\section{Deformation of \texorpdfstring{$\O$}--operators} \label{section4}

{We prove that each Maurer-Cartan element of a special graded Lie subalgebra of $\Coder(\bar S(E\oplus V))$ encondes a Lie $\infty$-algebra structure on $E$ and a curved Lie $\infty$-action of $E$ on $V$. We study deformations of $\O$-operators.}

\subsection{Maurer-Cartan elements of Coder
$\mathbf{(\bar{S}(E\oplus{V}))}$}

Let $E$ and $V$ be two graded vector spaces and consider the graded Lie algebra  $\mathfrak{L}:=( \Coder(\bar S(E\oplus V)), \brr{\cdot,\cdot}_{c})$.
Since
$\bar S(E\oplus V)\simeq \bar S(E) \oplus {(\bar S(E) \otimes \bar S(V))\oplus \bar S(V)}$,
the space $M:=\Coder(\bar S(E))$ of coderivations of $\bar S(E)$ can be seen as a graded Lie subalgebra of $\mathfrak{L}$. Also, the space $R$ of coderivations {defined by linear maps of the space ${\mathrm{Hom}}((\bar S(E) \otimes \bar S(V))\oplus \bar S(V),V)$} can be embedded in $\mathfrak{L}$.
We will use the identifications $M\equiv {\mathrm{Hom}}(\bar S(E),E)$ and $R\equiv  {\mathrm{Hom}}((\bar S(E) \otimes \bar S(V))\oplus \bar S(V),V)$. 
Given $\rho\in R$, we will denote by $\rho_0$ the restriction of the linear map $\rho$ to $\bar S(V)$ and by $\rho_x$ the linear map obtained by restriction of $\rho$ to $\set{x}\otimes \bar S(V)$, with $x \in \bar{S}(E)$.
We set $\mathfrak{L}':=M\oplus R$.
\begin{prop} \label{prop:L':Lie:subalgebra}
The space  $\mathfrak{L}'$ is a graded Lie subalgebra of $\mathfrak{L}=\Coder (\bar S(E \oplus V))$.
\end{prop}
\begin{proof}
Given $m\oplus \rho, m'\oplus \rho' \in \mathfrak{L}'$, let us see that
$$\brr{m\oplus \rho,m'\oplus \rho'}_{_{RN}}= \brr{m,m'}_{_{RN}} \oplus (\brr{m,\rho'}_{_{RN}} + \brr{\rho, m'}_{_{RN}} + \brr{\rho, \rho'}_{_{RN}})$$ is an element of $\mathfrak{L}'$. It is obvious that $\brr{m,m'}_{_{RN}}\in {\mathrm{Hom}}(\bar S(E),E)$. Consider $m^D$ and $\rho^D$ the coderivations of $\bar S(E\oplus V)$ defined by the morphisms $m$ and $\rho$, respectively. For $x \in \bar S(E)$ and $v \in \bar S(V)$ we have,
\begin{align*}
 \brr{m,\rho'}_{_{RN}}(x)&=\brr{m,\rho'}_{_{RN}}(v)=0\\
 \brr{m,\rho'}_{_{RN}}(x \otimes v)&= \left( m\smalcirc \rho'^D - (-1)^{|m||\rho'|} \rho' \smalcirc m^D \right)(x \odot v)\\
& = - (-1)^{|m||\rho'|} \rho'_{m^D(x)}(v) \in V
\end{align*}
and
\begin{align*}
 \brr{\rho,\rho'}_{_{RN}}(x)&=0\\
  \brr{\rho,\rho'}_{_{RN}}(v)&= \rho \smalcirc \rho'^D (v)- (-1)^{|\rho||\rho'|} \rho' \smalcirc \rho^D (v) \in V \\
 \brr{\rho,\rho'}_{_{RN}}(x \otimes v)&=
 \underbrace{(-1)^{|x||\rho'|}\rho_x(\rho_0'^D(v)) + (-1)^{|x_{(1)}||\rho'|}\rho_{x_{(1)}}(\rho_{x_{(2)}}'^D(v)) + \rho_0(\rho_x'^D(v))}_{\in V}\\
 &- (-1)^{|\rho||\rho'|}\big( \underbrace{(-1)^{|x||\rho|}\rho'_x(\rho_0^D(v))+ (-1)^{|x_{(1)}||\rho|}\rho'_{x_{(1)}}(\rho_{x_{(2)}}^D(v)) + \rho'_0(\rho_x^D(v)) }_{\in V}  \big),
 \end{align*}
which proves that $\brr{m,\rho'}_{_{RN}} + \brr{\rho, m'}_{_{RN}}+\brr{\rho, \rho'}_{_{RN}} \in {\mathrm{Hom}}((\bar S(E) \otimes \bar S(V))\oplus \bar S(V),V)$.
\end{proof}

\

Next theorem  shows that  an element $m \oplus \rho \in \mathfrak{L}'$ which is a Maurer-Cartan of $\mathfrak{L}=\Coder (\bar S(E \oplus V))$  encodes a Lie $\infty$-algebra structure on $E$ and  an action of $E$ on the Lie $\infty$-algebra $V$.

\begin{thm} \label{prop:MC:L'+h}
Let $E$ and $V$ be two graded vector spaces and $m \oplus \rho \in \mathfrak{L}'=M\oplus R$. Then,
$m \oplus \rho$ is a Maurer-Cartan element of $\mathfrak{L}'$   if and only if
$m^D$ defines a Lie $\infty$-structure on $E$ and $\rho$ is a curved Lie $\infty$-action of $E$ on $V$ .
\end{thm}
\begin{proof}

We have
\begin{equation} \label{eq:MC_equival_action}
\brr{m\oplus \rho,m\oplus \rho}_{_{RN}}=0 \Leftrightarrow
\begin{cases}\brr{m,m}_{_{RN}}=0\\ 2 \brr{m,\rho}_{_{RN}} + \brr{\rho, \rho}_{_{RN}}=0.
\end{cases}
\end{equation}

Similar computations to those in the proof of Proposition~\ref{prop:L':Lie:subalgebra} give, for all $v \in \bar S(V)$ and $x \in \bar S(E)$,
\begin{eqnarray*}
\lefteqn{\begin{cases}
\left( 2 \brr{m,\rho}_{_{RN}} + \brr{\rho, \rho}_{_{RN}}\right)(v)= 0 \\
\left( 2 \brr{m,\rho}_{_{RN}} + \brr{\rho, \rho}_{_{RN}}\right)(x\otimes v)=0
\end{cases}}\\
& \Leftrightarrow
 \begin{cases}
 \rho_0 \smalcirc \rho_0^D(v)=0\\
  \rho_{m^D(x)}(v) = \left( -\brr{\rho_0, \rho_x}_{_{RN}}
  -  \frac{(-1)^{|x_{(1)}|} }{2}\brr{\rho_{x_{(1)}}, \rho_{x_{(2)}}}_{_{RN}}\right)(v).
  \end{cases}
\end{eqnarray*}

Since $m \oplus \rho$ is a degree $+1$ element of $\mathfrak{L}'$, the  right hand-side of \eqref{eq:MC_equival_action}  means that $m^D$ defines a Lie $\infty$-algebra structure on $E$ and $\rho= \sum_{k\geq 0} \rho_k$ is a  curved Lie $\infty$-action of $E$ on $V$.
Notice that $\rho_0^D:\bar S(V) \to \bar S(V)$ equips $V$ with a Lie $\infty$-structure.

Reciprocally, if $(E,m^D)$ is a Lie $\infty$-algebra and $\rho$ is a curved Lie $\infty$-action of $E$ on $V$, the degree $+1$ element $m \oplus \rho$ of $\mathfrak{L}'$ is a Maurer-Cartan element of $\mathfrak{L}'$.
\end{proof}

Next proposition gives the Lie $\infty$-algebra that controls the deformations of the actions of $E$ on $V$ \cite{G}.

\begin{prop}  \label{prop:Lie:infty:controls:MC}
Let  $m \oplus \rho$ be a Maurer-Cartan element of $\mathfrak{L}'$ and $m' \oplus \rho'$ a degree $+1$ element of $\mathfrak{L}'$. Then, $m \oplus \rho+ m' \oplus \rho'$ is a Maurer-Cartan element of $\mathfrak{L}'$ if and only if $m' \oplus \rho'$ is a Maurer-Cartan element of $\mathfrak{L}'\,^{m \oplus \rho}$.
Here, $\mathfrak{L}'\,^{m \oplus \rho}$ denotes the DGLA which is the twisting of $\mathfrak{L}'$ by $m \oplus \rho$.
\end{prop}


\subsection{Deformation of \texorpdfstring{$\O$}--operators}
Let $\mathfrak{h}$ be the abelian Lie subalgebra of $\mathfrak{L}=\Coder(\bar S(E\oplus V))$ considered in Proposition~\ref{prop:h:abelian:subalgebra} and $\mathcal{P}:\mathfrak{L} \to \mathfrak{h}$ the projection onto $\mathfrak{h}$. Let $\Delta \in \mathfrak{L}'$ be a Maurer-Cartan element of $\mathfrak{L}$. Then,  $\ds \left(\mathfrak{L}, \mathfrak{h},  \mathcal{P}, \Delta \right)$  is a $V$-data and $\mathfrak{h}$ has a Lie $\infty$-structure given by the brackets:
$$\partial_{k}(a_{1},\ldots,a_{k})=\mathcal{P}([[\ldots\brr{\Delta, a_{1}}_{_{RN}}\ldots]_{_{RN}}, a_{k}]_{_{RN}}), \quad k\geq 1.$$
We denote by $\mathfrak{h}_\Delta$ the Lie $\infty$-algebra $\mathfrak{h}$ equipped with the above  structure.

The $V$-data $\ds \left(\mathfrak{L}, \mathfrak{h},  \mathcal{P}, \Delta \right)$ also provides a Lie $\infty$-algebra structure on $\mathfrak{L}[1] \oplus \mathfrak{h}$, that we denote by $(\mathfrak{L}[1] \oplus \mathfrak{h})_\Delta$, with brackets \cite{V05}:

\begin{equation} \label{eq:bracket:q:Delta}
\begin{cases}
q_1^\Delta((x,a_1))=(-\brr{\Delta, x}_{_{RN}}, \mathcal{P}(x+ \brr{\Delta, a_1}_{_{RN}})) \vspace{.2cm}\\
q_2^\Delta(x,x')= (-1)^{deg(x)}\brr{x,x'}_{_{RN}}\vspace{.2cm} \\
q_k^\Delta(x, a_1, \ldots, a_{k-1})=\mathcal{P}([ \ldots [[x, a_1]_{_{RN}}, a_2]_{_{RN}}\ldots a_{k-1}]_{_{RN}}),\,\,\,k\geq2,\vspace{.2cm} \\
q_k^\Delta(a_1, \ldots, a_{k})= \partial_k(a_1, \ldots, a_{k}),\,\,\,k\geq 1,
\end{cases}
\end{equation}
$x,x' \in \mathfrak{L}[1]$ and $a_1, \ldots,a_{k-1} \in \mathfrak{h}$. Here, $deg(x)$ is the degree of $x$ in $\mathfrak{L}$.

\

Moreover, since $\mathfrak{L}'$ is a Lie subalgebra of $\mathfrak{L}$ satisfying $\brr{\Delta, \mathfrak{L}'}\subset \mathfrak{L}'$, the brackets $\set{q_k^\Delta}_{k \in \Nn}$ restricted to $\mathfrak{L}'[1] \oplus \mathfrak{h}$  define a Lie $\infty$-algebra structure on  $\mathfrak{L}'[1] \oplus \mathfrak{h}$, that we denote by $(\mathfrak{L}'[1] \oplus \mathfrak{h})_\Delta$. Notice that the restrictions of the brackets $\set{q_k^\Delta}$ to $\mathfrak{L}'[1] \oplus \mathfrak{h}$ are given by the same expressions as in \eqref{eq:bracket:q:Delta} except for $k=1$:
$$q_1^\Delta((x,a_1))=(-\brr{\Delta, x}_{_{RN}}, \mathcal{P}(\brr{\Delta, a_1}_{_{RN}}))=(-\brr{\Delta, x}_{_{RN}}, \partial_1(a_1)),$$
because $\mathcal{P}(\mathfrak{L}')=0$.
Of course, $\mathfrak{h}_\Delta$ is a Lie $\infty$-subalgebra of $(\mathfrak{L}'[1]\oplus \mathfrak{h})_\Delta$.

\begin{rem}
The brackets \eqref{eq:bracket:q:Delta} that define the Lie $\infty$-algebra structure of $(\mathfrak{L}[1] \oplus \mathfrak{h})_\Delta$ coincide with those of $\mathfrak{h}_\Delta$ for $x=x'=0$ . So, an easy computation yields
$$t \in \textrm{MC}(\mathfrak{h}_\Delta) \, \Leftrightarrow \, (0,t) \in  \textrm{MC}(\mathfrak{L}'[1] \oplus \mathfrak{h})_\Delta.$$
\end{rem}

Theorem 3 in \cite{FZ} yields:

\begin{prop} \label{prop_FZ}
Consider the $V$-data $\ds \left(\mathfrak{L}, \mathfrak{h},  \mathcal{P}, \Delta \right)$, with $\Delta\in \emph{MC}(\mathfrak{L}')$ and let $t$ be a degree zero element of $\mathfrak{h}$. Then,
$$t \in \emph{MC}(\mathfrak{h}_{\Delta})
\, \Leftrightarrow \,
(\Delta,t)\in \emph{MC}(\mathfrak{L}[1]\oplus \mathfrak{h})_{\Delta}.$$
\end{prop}

Recall that, given an element $t \in \mathfrak{h}={\mathrm{Hom}}(\bar S(V), E)$, the corresponding morphism of coalgebras $T: \bar S(V) \to \bar S(E)$ is an $\O$-operator if and only if $t$ is a Maurer-Cartan element of $\mathfrak{h}_\Delta$ (Proposition~\ref{prop:1-1correspondence}). Moreover, given a Maurer-Cartan element $m \oplus \rho$ of $\mathfrak{L'}$, we know from Theorem \ref{prop:MC:L'+h} that $(E, m^D)$ is a Lie $\infty$-algebra and $\rho$ is a curved Lie $\infty$-action of $E$ on $V$. So,
 an $\O$-operator can be seen as a Maurer-Cartan element of  the Lie $\infty$-algebra $(\mathfrak{L}'[1] \oplus \mathfrak{h})_\Delta$:

\begin{prop}  \label{prop:T:MC}
Let $E$ and $V$ be two graded vector spaces. Consider a morphism of coalgebras $T: \bar S(V) \to \bar S(E)$ defined by $t \in \mathrm{Hom}(\bar S(V), E)$, and the $V$-data $\ds \left(\mathfrak{L}, \mathfrak{h},  \mathcal{P}, \Delta \right)$, with $\Delta:=m \oplus \rho \in  \emph{MC}(\mathfrak{L'})$. Then,
$T$ is an $\O$-operator  on $E$ with respect to  the  curved Lie $\infty$-action $\rho$ if and only if $(\Delta, t)$ is a Maurer-Cartan element of $(\mathfrak{L}'[1] \oplus \mathfrak{h})_\Delta$.
\end{prop}

\begin{cor}
If $T$ is an $\O$-operator on the Lie $\infty$-algebra $(E, m^D)$ with respect to the curved Lie $\infty$-action $\rho$ of $E$ on $V$, then $((\mathfrak{L}'[1] \oplus \mathfrak{h})_{m \oplus \rho})^{(m \oplus \rho,t)}$ is a Lie $\infty$-algebra.
\end{cor}

As a consequence of Theorem 3 in \cite{FZ}, we obtain the Lie $\infty$-algebra that controls the deformation of $\O$-operators on $E$ with respect to a   fixed curved Lie $\infty$-action on $V$:

 \begin{cor} Let $E$ and  $V$ be two  graded vector spaces and consider the $V$-data $(\mathfrak{L}, \mathfrak{h}, \mathcal{P}, \Delta:=m\oplus \rho)$.
Let $T$ be an $\O$-operator on $(E,m^D)$ with respect to the curved Lie $\infty$-action $\rho$ and $T':\bar S(V) \to \bar S(E)$ a (degree zero) morphism of coalgebras defined by $t' \in \mathrm{Hom}(\bar S(V), E)$. Then, $T+T'$ is an $\O$-operator on $E$ with respect to the curved Lie $\infty$-action $\rho$ if and only if $(\Delta, t')$ is a Maurer-Cartan element of $(\mathfrak{L}'[1]\oplus \mathfrak{h})_{\Delta}^{(\Delta, t)}$.
\end{cor}
\begin{proof}
Let $t \in \mathfrak{h}$ be the morphism defined by $T$. Then \cite{FZ},
$$(\Delta, t+t')\in \textrm{MC}(\mathfrak{L}'[1]\oplus \mathfrak{h})_{\Delta} \, \Leftrightarrow  (\Delta,t')\in \textrm{MC}(\mathfrak{L}'[1]\oplus \mathfrak{h})_{\Delta}^{(\Delta, t)}.
$$
\end{proof}


\end{document}